\documentclass[12pt]{amsart}
\usepackage[utf8]{inputenc}	
\usepackage{amssymb,amscd}
\usepackage{multirow,booktabs}
\usepackage[table]{xcolor}
\usepackage{fullpage}
\usepackage{mathtools}
\usepackage{lastpage}
\usepackage{enumitem}
\usepackage{fancyhdr}
\usepackage{mathrsfs}
\usepackage{wrapfig}
\usepackage{graphicx}
\usepackage{setspace}
\usepackage{calc}
\usepackage{xfrac}
\usepackage{pifont}
\usepackage{subfigure}
\usepackage{multicol}
\usepackage[colorlinks,
            linkcolor=blue,
            anchorcolor=blue,
            citecolor=red,
            ]{hyperref}
\usepackage{cancel}
\usepackage[all]{xy}
\usepackage[margin=3cm]{geometry}

\usepackage{empheq}
\usepackage{framed}
\usepackage{dsfont}
\usepackage{xcolor}

\theoremstyle{plain}
\newtheorem{theorem}{Theorem}[section]
\newtheorem{lemma}[theorem]{Lemma}
\newtheorem{proposition}[theorem]{Proposition}
\newtheorem{corollary}[theorem]{Corollary}

\theoremstyle{remark}
\newtheorem{definition}{Definition}
\newtheorem{example}{Example}[section]
\newtheorem{remark}{Remark}[section]

\usepackage{indentfirst}
\begin{document}
\setcounter{section}{0}

\thispagestyle{empty}

\newcommand{\QQ}{\mathbb{Q}}
\newcommand{\RR}{\mathbb{R}}
\newcommand{\ZZ}{\mathbb{Z}}
\newcommand{\NN}{\mathbb{N}}
\newcommand{\Nor}{\mathscr{N}}
\newcommand{\CC}{\mathbb{C}}
\newcommand{\HH}{\mathbb{H}}
\newcommand{\EE}{\mathbb{E}}
\newcommand{\Var}{\operatorname{Var}}
\newcommand{\PP}{\mathbb{P}}
\newcommand{\Rd}{\mathbb{R}^d}
\newcommand{\Rn}{\mathbb{R}^n}
\newcommand{\XX}{\mathcal{X}}
\newcommand{\YY}{\mathcal{Y}}
\newcommand{\MM}{\FF}
\newcommand{\BHH}{\overline{\mathbb{H}}}
\newcommand{\XB}{( \mathcal{X},\mathscr{B} )}
\newcommand{\BB}{\mathscr{B}}
\newcommand{\system}{(\Omega,\mathcal{F},\mu,T)}
\newcommand{\FF}{\mathcal{F}}
\newcommand{\MBS}{(\Omega,\mathcal{F})}
\newcommand{\MBSE}{(E,\mathscr{E})}
\newcommand{\MS}{(\Omega,\mathcal{F},\mu)}
\newcommand{\PS}{(\Omega,\mathcal{F},\mathbb{P})}
\newcommand{\LDP}{LDP(\mu_n, r_n, I)}
\newcommand{\Def}{\overset{\text{def}}{=}}
\newcommand{\Series}[2]{#1_1,\cdots,#1_#2}
\newcommand{\independent}{\perp\mkern-9.5mu\perp}
\def\avint{\mathop{\,\rlap{-}\!\!\int\!\!\llap{-}}\nolimits}

\author[Shuo Qin]{Shuo Qin}
\address[Shuo Qin]{Courant Institute of Mathematical Sciences \& NYU-ECNU Institute of Mathematical Sciences at NYU Shanghai}
\email{sq646@nyu.edu}

\title{Recurrence and transience of multidimensional elephant random walks}

\date{}

\begin{abstract}
   We prove a conjecture by Bertoin \cite{bertoin2022counting} that the multi-dimensional elephant random walk on $\ZZ^d$($d\geq 3$) is transient and the expected number of zeros is finite. We also provide some estimates on the rate of escape. In dimensions $d= 1, 2$, we prove that phase transitions between recurrence and transience occur at $p=(2d+1)/(4d)$. 
   
   Let $S$ be an elephant random walk with parameter $p$. For $p \leq 3/4$, we provide a new version of Berry-Esseen type bound for properly normalized $S_n$. For $p>3/4$, the distribution of $\lim_{n\to \infty} S_n/n^{2p-1}$ will be studied.
  
  \end{abstract}

\maketitle

\section{General introduction}

\subsection{Definitions and main results}
\label{secdefmainre}
The elephant random walk(ERW) was introduced by Schütz and Trimper \cite{schutz2004elephants} to investigate the long-term memory effects in non-Markovian random walks. The multi-dimensional elephant random walk(MERW), which is the natural extension to a higher dimension of the ERW, is a nearest-neighbor random walk on $\mathbb{Z}^d$ $(d\geq 1)$, studied by Bercu, Laulin in \cite{bercu2019multi} and Bertenghi in \cite{bertenghi2022functional}. When $d=2$, similar models were studied in \cite{cressoni2013exact} and \cite{lyu2017residual}.

The aim of this paper is to solve some open problems raised, among others, by Bertoin \cite{bertoin2022counting} on the MERW.

For $d\geq 1$, let $e_1, \ldots, e_d$ denote the standard basis for the Euclidean space $\mathbb{R}^d$. Define a random walk $(S_n)_{n\in \NN}$ on $\ZZ^d$ as follows. First let $S_0=0 \in \ZZ^d$. Without loss of generality, we assume that $S_1=\sigma_1=e_1$ unless otherwise specified. At any time $n \geq 1$, we choose a number $n^{\prime}$ uniformly at random among the previous times $1, \ldots, n$ and set $M_{n^{\prime}}:=\left\{\pm e_1, \ldots, \pm e_d\right\} \backslash\left\{\sigma_{n^{\prime}}\right\}$. We then define a random vector $\sigma_{n+1}$ by 
\begin{equation}
  \label{lawMERWsigma}
  \mathbb{P}\left(\sigma_{n+1}=\sigma_{n^{\prime}}\right)=p, \quad \mathbb{P}\left(\sigma_{n+1}=\sigma\right)=\frac{1-p}{2 d-1} \quad \text { for all } \sigma \in M_{n^{\prime}}
\end{equation}
where $p \in [0,1]$. Now set
$$
S_{n+1}:=S_n+\sigma_{n+1}
$$
Then $(S_n)_{n\in \NN}$ is called a multi-dimensional elephant random walk(MERW) on $\ZZ^d$ with memory parameter $p$. When $d=1$, $S$ is called a one-dimensional ERW(or simply called an ERW) with memory parameter $p$. 
\begin{remark}
  \label{remdefMERW}
 (i) By definition, a MERW $S$ with parameter $p=1/(2d)$ is the simple random walk on $\ZZ^d$ with $S_1=e_1$. \\
 (ii) The ERW is related to a bounded rationality model in economics \cite{MR1268114}: suppose we have two competing technologies, at each time step, a new agent randomly asks one previous agent which technology he/she is using. Then the agent selects with probability $p$ the technology used by the previous agent and with probability $1-p$ the other one. The case $p=0$ has been introduced and considered by Arthur et al. \cite{MR0720758}. \\
 (iii) The (multi-dimensional) elephant random walk has been a fundamental example of the so-called step-reinforced random walks, see e.g. \cite{MR4116724}, \cite{bertoin2021universality}: Fix a parameter $a \in[0,1]$ and a distribution $\mu$ on $\mathbb{Z}^d$. At each time step, with probability $a$, a step-reinforced random walk repeats one of its preceding steps chosen uniformly at random, and otherwise, with complementary probability $1-a$, it has an independent increment with distribution $\mu$. As was pointed out in \cite{MR3652690} for the case $d=1$, if $\mu$ is the uniform measure on $\left\{ \pm e_1, \ldots, \pm e_d\right\}$, the step-reinforced random walk is a version of the MERW on $\mathbb{Z}^d$ with parameter $p=\frac{(2 d-1) a+1}{2 d}$.
\end{remark}

\begin{definition}
  A random walk $\left(S_n\right)_{n \in \mathbb{N}}$ on $\mathbb{Z}^d$, $d\geq 1$, is said to be recurrent, resp. transient, if for all $i \in \mathbb{Z}^d$, 
  $$\mathbb{P}(S_n=i,\ \text{infinitely often})=1, \quad \text{resp.}\quad  \mathbb{P}(S_n=i, \ \text{finitely often})=1$$ 
\end{definition}

In dimensions $d\geq 3$, Bertoin \cite{bertoin2022counting} conjectured that the expected number of zeros of a MERW $(S_n)_{n\in \NN}$, i.e. $\EE( \sum_{n=1}^{\infty}\mathds{1}_{\{S_n=0\}})$, is finite and, a fortiori, the walk is transient. We prove that conjecture and give a lower bound for the rate of escape. We write $\|x\|:=\|x\|_{L^2}$ for $x\in \ZZ^d$.
\begin{theorem}
 \label{MERWtran12d}
 Let $S=(S_n)_{n\in \NN}$ be a MERW on $\ZZ^d$ with $d\geq 3$ and memory parameter $p$. Then, for any $p \in [0,1]$, $S$ is transient. More precisely, for any $\nu \in (0,\frac{1}{2}-\frac{1}{d})$,
 $$
 \EE \left( \sum_{n=1}^{\infty}\mathds{1}_{\{\|S_n\|\leq n^{\nu}\}}\right) < \infty  
  $$
  in particular, almost surely, $\|S_n\|> n^{\nu}$ for large $n$. 
\end{theorem}

In the case $p\geq 1/(2d)$, we give a sharper lower bound.

\begin{proposition}
    \label{rateescapMERW2d}
 Let $S$ be a MERW on $\ZZ^d$ with $d\geq 3$ and memory parameter $p\geq 1/(2d)$. Then, almost surely, for all but finitely many $n> 1$, 
 $$
\|S_n\| \geq n^{\frac{1}{2}}(\log n)^{-3}
$$
\end{proposition}

The two-dimensional case is considered as a challenging problem by Bertoin \cite{bertoin2022counting}, as even in the diffusive regime, one cannot conclude that recurrence holds because of the failure of the Markov property, and a finer analysis is required. We prove that there is a phase transition between recurrence and transience at $p=5/8$. 

\begin{theorem}
  \label{phasetranZ2MERW}
  The MERW on $\ZZ^2$ with memory parameter $p$ is recurrent for $p<5/8$ and transient for $p \geq 5/8$.
\end{theorem}
\begin{remark}
Recently, Curien and Laulin \cite{curien2023recurrence} gave a new proof of the recurrence for $p<5/8$.
\end{remark}

In dimension 1, we retrieve a result by Coletti and Papageorgiou \cite{coletti2021asymptotic} on the phase transition at $p=3/4$. Note that their proof is incorrect, as we explain in Remark \ref{p34proincorrect}.

\begin{theorem}
  \label{phasetranERWZ}
  The ERW on $\ZZ$ with parameter $p$ is recurrent for $p\leq 3/4$ and transient for $p>3/4$.
\end{theorem}
\begin{remark}
Bertoin proved in Theorem 4.1 \cite{bertoin2022counting} that the ERW is positive recurrent $($i.e. the expectation of $\inf\{n\geq 1: S_n=0\}$ is finite$)$ if and only if $p<1/4$. 
\end{remark}

The recurrence for $p\leq 3/4$ could already be deduced from the law of the iterated logarithm for the ERW by Bercu \cite{bercu2017martingale} and Coletti, Gava and Schütz \cite{coletti2017strong}, as was pointed out by Coletti and Papageorgiou in \cite{coletti2021asymptotic}. 

Also, for a MERW $S$ on $\ZZ^d$, it has been shown by Bercu \cite{bercu2017martingale} and Coletti, Gava and  Schütz \cite{coletti2017central} in dimension 1, Bercu and Laulin \cite{bercu2019multi} in any dimension $d\geq 1$ that if $p>(2d+1)/(4d)$, then
\begin{equation}
  \label{superlimY}
  \lim _{n \rightarrow \infty} \frac{ S_n}{n^{a(d,p)}}=Y_d \quad a.s., \ \text{with}\ a(d,p):=\frac{2 d p-1}{2 d-1}
\end{equation}
where the limit $Y_d$ is a random vector such that $\PP(Y_d=0)<1$. In particular, $S$ goes to infinity with positive probability for $p>(2d+1)/(4d)$. We write $a=a(d,p)$ when there is no ambiguity and let
\begin{equation}
  \label{pddef}
  p_d:=\frac{2d+1}{4d}
\end{equation}
Note that $a<1/2 \Leftrightarrow p<p_d$, $a=1/2 \Leftrightarrow p=p_d$, and  $a>1/2 \Leftrightarrow p> p_d$

Our proof of transience in Theorem \ref{phasetranZ2MERW} and Theorem \ref{phasetranERWZ} relies on those results, as we show that for any $p>p_d$, $Y_d\neq 0$ a.s.. 

\begin{proposition}
   \label{Lnot0MERW}
  Let $S$ be a MERW on $\ZZ^d$ with parameter $p >p_d$. Then, $\PP(Y_d=0)=0$, where $Y_d$ is defined in (\ref{superlimY}). In particular, $S$ is transient.
\end{proposition}
\begin{remark}
  \label{fixedpointpaper}
 (i) A work of Gu{\'e}rin, Laulin and Raschel \cite{guerin2023fixed} appeared very recently on the arXiv, showing independently, using a similar approach, that, for $d=1,2,3$, $Y_d$ admits a density, which obviously implies $\PP(Y_d = 0) = 0$ in those dimensions, and in particular, Theorem 1.4. 
 
 (ii) As is noted in the introduction part of \cite{bercu2022estimate}, $\PP(Y_1=0)=0$ could already be deduced from the works of Bertoin, Baur \cite{MR3399834} and Businger \cite{MR3827299}. See also Proposition 2.15 in \cite{guerin2023fixed}. To show that $\PP(Y_d=0)=0$, it may be possible to follow a similar technique by using Equation (3.5) in \cite{guerin2023fixed}, but here we adopt a continuous-time embedding method.
\end{remark}

For a MERW $S$ on $\ZZ^2$ with $p=5/8$, we prove the following almost-sure convergence. 

\begin{proposition}
  \label{z258logn}
  Let $S$ be a MERW on $\ZZ^2$ with parameter $p=5/8$. Then, 
  \begin{equation}
    \label{logS2lognlim1}
     \lim_{n\to \infty} \frac{\log \|S_n\|^2}{\log n}=1, \quad a.s.
  \end{equation}
\end{proposition}
\begin{remark}
 (i) The law of the iterated logarithm for a MERW $S$ on $\ZZ^d$ with parameter $p\leq p_d$ was proved by Bercu and Laulin in \cite{bercu2019multi}: for $p<p_d$,
  \begin{equation}
  \label{ltilSdiff}
  \limsup _{n \rightarrow \infty} \frac{\left\|S_n\right\|^2}{2 n \log \log n}=\frac{1}{1-2 a} \quad \text { a.s. }
\end{equation}
and for $p=p_d$, 
\begin{equation}
  \label{ltilScrit}
  \limsup _{n \rightarrow \infty} \frac{\left\|S_n\right\|^2}{2 n \log n \log \log \log n}=1 \quad \text { a.s. }
\end{equation}
In the setting of Proposition \ref{z258logn}, (\ref{ltilScrit}) implies that
$$
\limsup_{n\to \infty} \frac{\log \|S_n\|^2}{\log n}=1, \quad a.s.
$$
(ii) By (\ref{ltilSdiff}), (\ref{ltilScrit}) and Proposition \ref{rateescapMERW2d}, (\ref{logS2lognlim1}) also holds for $d\geq 3$ and $p\in [1/(2d),p_d]$.
\end{remark}

The recurrence and transience properties described above are properties of the radial component of the MERW. It is also natural to investigate the asymptotics of the angular component. We say that a transient random walk $S$ on $\ZZ^d$ has a limiting direction if $\lim _{n \rightarrow \infty} \hat{S}_n$ exists in the unit sphere $\mathbb{S}^{d-1}$. Here $\hat{x}$ denotes the unit vector $x /\|x\|$ for $x \in \RR^d \backslash \{0\}$. By Proposition \ref{Lnot0MERW}, the MERW has a limiting direction if $p>p_d$, as illustrated in Figure \ref{Fig1-2}. On the other hand, if $p\leq p_d$, we show that almost surely, the limiting direction does not exist, as illustrated in Figure \ref{Fig1-1}.

\begin{corollary}
  \label{angularasym}
  Let $S$ be a MERW on $\ZZ^d$ with parameter $p$. 
  \begin{enumerate}[topsep=0pt, partopsep=0pt, leftmargin=5pt, align=left,  label=(\roman*)]
 \item If $p>p_d$, then $\lim _{n \rightarrow \infty} \hat{S}_n$ exists a.s. 
 \item If $d\geq 3$ and $p \leq p_d$, or $d=2$ and $p=p_d$, then $\PP(\lim _{n \to \infty} \hat{S}_n\ \text{exists})=0$.
  \end{enumerate}
\end{corollary}

\begin{figure}[t]
  \centering
  \subfigure[p=0.5]{
 \includegraphics[width=7.7cm]{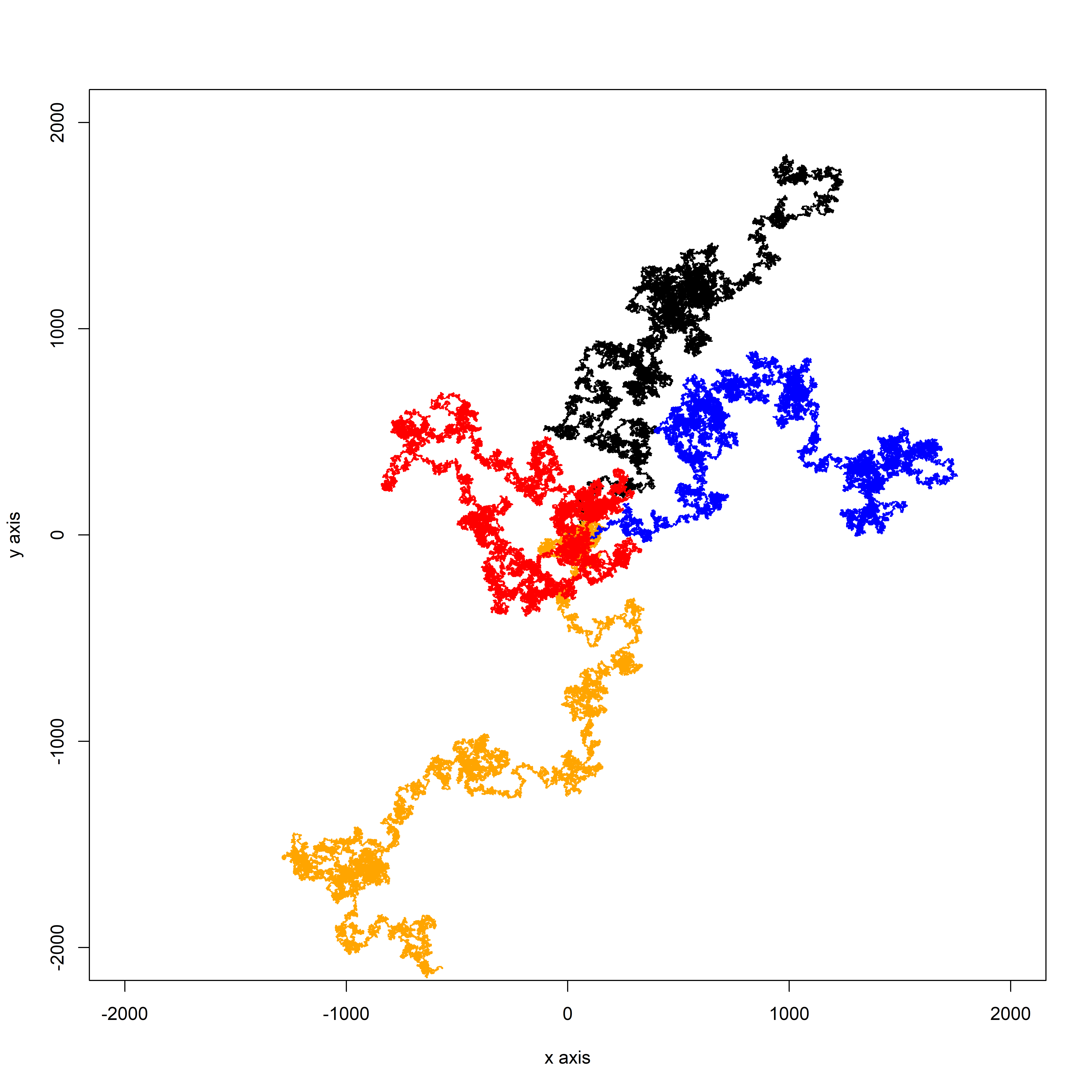}
 \label{Fig1-1}
 }
  \quad
 \subfigure[p=0.7]{
  \includegraphics[width=7.7cm]{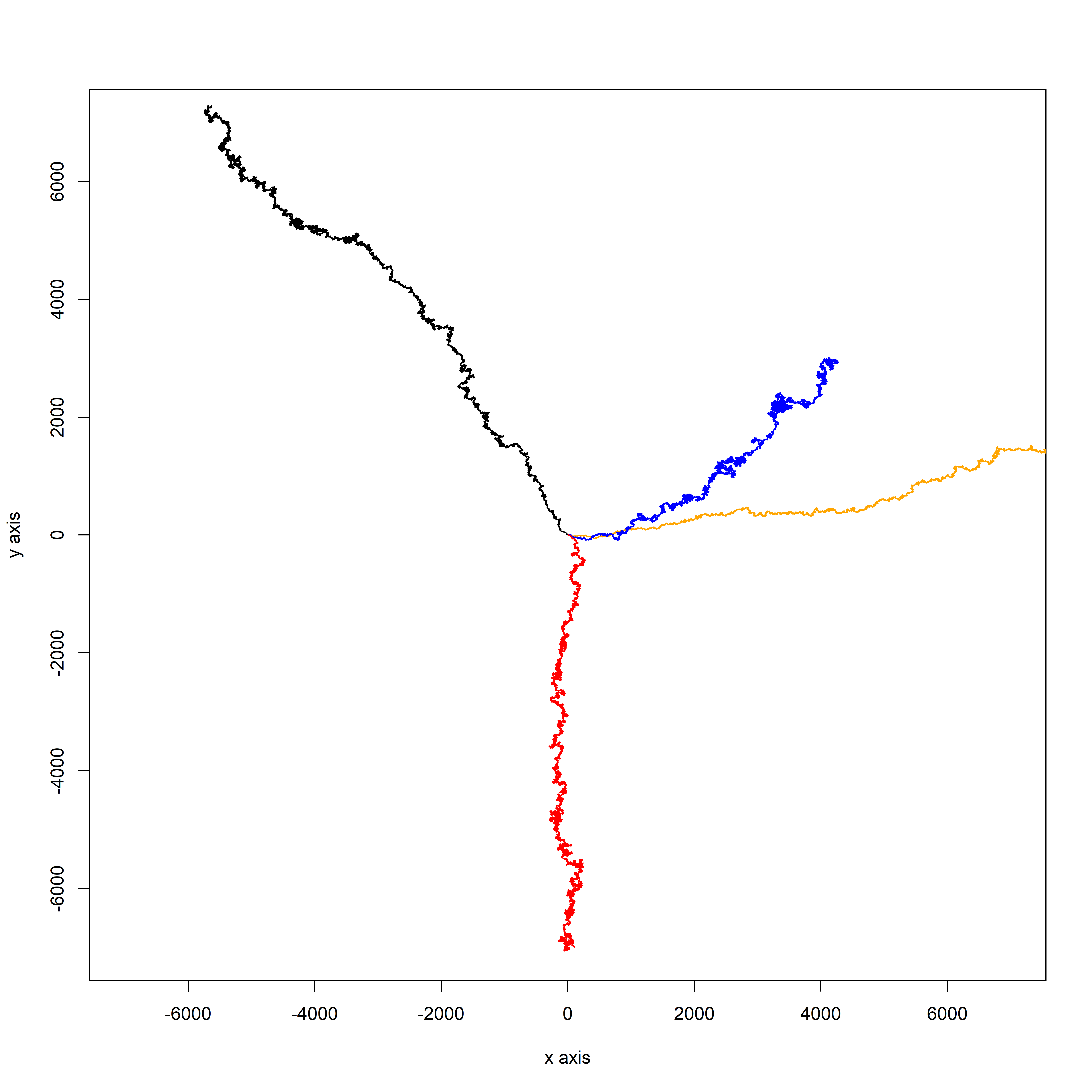}
  \label{Fig1-2}
  }
  \caption{Simulation results of 4 independent MERWs on $\ZZ^2$ with $n=1000000$ steps, showing the range of the walks}
  \label{simuMERW2}
\end{figure} 

\subsection{Introduction to the proofs and the techniques}
\label{sectechniques}

\subsubsection{Notation} Throughout this paper, we let $(\FF_n)_{n\geq 1}$ denote the filtration generated by the MERW $S$ we are studying. We let $x(i)$ denote the $i$-th coordinate of a vector $x$. In particular, we may write $S_n=(S_n(1),S_n(2),\cdots,S_n(d))$ for a MERW $(S_n)_{n\in \NN}$ on $\ZZ^d$. We write $S_n(i)^2:=(S_n(i))^2$ for simplicity. We let $C(a_1,a_2,\cdots,a_k)$ denote a positive constant depending only on real variables $a_1, a_2, \ldots, a_k$ and let $C$ denote a universal positive constant, which usually means that $C(a_1,a_2,\cdots,a_k)$ and $C$ do not depend on $n$. We denote the number of ways one can choose $m$ objects from a collection of $n$ objects by $C_n^m=n !/(m! (n-m)!)$. For a real-valued function $h$ and a $[0,\infty)$-valued function $g$, we write $h(x)=O(g(x))$ as $x\to \infty$, resp. $x\to 0$, if there exist positive constants $C$ and $x_0$ such that $|h(x)| \leq C g(x)$ for all $x \geq x_0$, resp. $|x|\leq x_0$. Given $x\in \RR$, we let $x^{-}=\max (-x, 0)$. Given two sequences $\left(u_n\right)_{n \in \mathbb{N}}$ and $\left(v_n\right)_{n \in \mathbb{N}}$ taking values in $\mathbb{R}$, we write $u_n \sim v_n$ if $u_n/v_n \to 1$ as $n$ goes to infinity. The tilde can also mean “is distributed as", e.g. we write $X \sim \operatorname{Exp}(\lambda)$ if a random variable $X$ has an exponential distribution with parameter $\lambda$. We write $X \stackrel{\mathcal{L}}{=} Y$ if two random variables/processes $X$ and $Y$ have the same distribution. $\mathcal{L}(X)$ and $\mathcal{L}(X|\cdot)$ denote the law and conditional law of $X$, respectively, where $X$ is a random variable/process.

For any time $n \geq 1$, $\sigma_n:=S_n-S_{n-1}$ denotes the n-th step of the MERW $S$. For $i\in \{1,2,\cdots,d\}$, we let $N_n(i)$, resp. $N_n(-i)$, be the number of steps of $S$ in the direction $e_i$, resp. $-e_i$, up to time $n$, i.e.
\begin{equation}
  \label{defNni}
    N_n(\pm i):= \#\{1\leq j \leq n: \sigma_j= \pm e_i\}
\end{equation}
Then, by definition,
\begin{equation}
  \label{Psigman1}
  \PP(\sigma_{n+1} = \pm e_i|\FF_n)=\frac{pN_n(\pm i)}{n}+\frac{1-p}{2 d-1}\left(1-\frac{N_n(\pm i)}{n}\right)=\frac{2dp-1}{2d-1}\frac{N_n(\pm i)}{n}+\frac{1-p}{2d-1}.
\end{equation}
Using that $S_n(i)=N_n(i)-N_{n}(-i)$ and $S_n=\sum_{i=1}^dS_n(i)e_i$, we have 
  \begin{equation}
    \label{Esigman1}
     \EE (\sigma_{n+1}|\FF_n)= \sum_{i=1}^d e_i \left(\PP(\sigma_{n+1} = e_i|\FF_n)-\PP(\sigma_{n+1} = -e_i|\FF_n)\right)  =\frac{2dp-1}{2d-1} \frac{S_n}{n}=\frac{a S_n}{n}.
  \end{equation}
  For any $n\geq 1$ and $i\in \{1,2,\cdots,d\}$, let 
\begin{equation}
  \label{numiaxis}
   b_n(i):= \#\{1\leq j \leq n: \sigma_j\in \{e_i,-e_i\}\}
\end{equation}
be the number of steps of $S$ along the i-th axis up to time $n$ and let
\begin{equation}
  \label{defcin}
  c_n(i):=\frac{2dp-1}{2d-1} \frac{b_n(i)}{n}+\frac{2-2p}{2d-1} 
\end{equation}

The following lemma computes the probability at time $n$ to move on the i-th coordinate, as well as the average move in that direction and its variance. Similar computations were done in Section 4 \cite{bercu2019multi}.
\begin{lemma}
  \label{computationlem}
  Let $S$ be a MERW on $\ZZ^d$ with parameter $p$. Let $b_n(i),c_n(i)$ be defined as in (\ref{numiaxis}) and (\ref{defcin}). Then, 
  \begin{enumerate}[topsep=0pt, partopsep=0pt, leftmargin=5pt, align=left,  label=(\roman*)]
\item For any $n\geq 1$ and $i\in \{1,2,\cdots,d\}$, $\PP(\sigma_{n+1}= e_i \ \text{or}\ -e_i|\FF_n)=c_n(i)$.
\item Recall $a=a(d,p)$ from (\ref{superlimY}). For any $n\geq 1$, $\EE (S_n\cdot \sigma_{n+1}|\FF_n)=  a\|S_n\|^2/n$.
\item  For any $n\geq 1$, we have
\begin{equation}
  \label{condsecmo}
    \EE ([S_n\cdot \sigma_{n+1}]^2|\FF_n)=\frac{1}{d}\|S_n\|^2+ \sum_{i=1}^d a(\frac{b_n(i)}{n}-\frac{1}{d})S_n(i)^2
 \end{equation}
  \end{enumerate}  
\end{lemma}
\begin{proof}[Proof of Lemma \ref{computationlem}]
  (i) Note that $b_n(i)=N_n(i)+N_n(-i)$. Then (\ref{Psigman1}) implies (i). (ii) follows from (\ref{Esigman1}). Moreover, we deduce from (i) that
  $$
  \EE ([S_n\cdot \sigma_{n+1}]^2|\FF_n)=\sum_{i=1}^dc_n(i)S_n(i)^2=\frac{1}{d}\|S_n\|^2+ \sum_{i=1}^d a(\frac{b_n(i)}{n}-\frac{1}{d})S_n(i)^2
  $$
  which completes the proof of (iii).
\end{proof}

Now we introduce the main techniques we use to prove the theorems in Section \ref{secdefmainre}.

\subsubsection{Coupling} The main technique we use to prove Theorem \ref{MERWtran12d} is a coupling method. We start with a statement in dimension 1 showing that for each sample path on some probability space, ERW with a larger parameter, is further away from the origin.

\begin{proposition}
  \label{monotonicitythm}
Let $0\leq p_1\leq p_2 \leq 1$. We can define two nearest-neighbor random walks $S$ and $\tilde{S}$ on the same probability space $\PS$ such that 
\begin{enumerate}[topsep=0pt, partopsep=0pt,leftmargin=5pt, align=left,  label=(\roman*)]
\item $S$ and $\tilde{S}$ are ERWs with parameters $p_1$ and $p_2$, respectively, 
\item for any $n\geq 1$, $|S_n| \leq |\tilde{S}_n|$. In particular, $S_n=0$ if $\tilde{S}_n=0$.
\end{enumerate}
\end{proposition}

For a MERW $S$, we denote by $\zeta_m$ the exit time of $S$ from $B(0,m)$ $(m\in \NN\backslash \{0\})$, i.e.
\begin{equation}
  \label{exittimedef}
  \zeta_m:=\inf\{n\in \NN: \|S_n\|\geq m\}.
\end{equation}

\begin{corollary}
  \label{erwsrwcouple}
      Let $S$ be an ERW with parameter $p$, and let $m$ be a positive integer.  
      \begin{enumerate}[topsep=0pt, partopsep=0pt, leftmargin=5pt, align=left,  label=(\roman*)]
      \item  If $p\leq 1/2$, then $S$ is recurrent and $\EE \zeta_m \geq m^2$. Moreover,
  $$
 \limsup_{n\to \infty} \frac{|S_n|}{ \sqrt{2n \log \log n}} \leq 1, \ a.s. \quad \liminf_{n\to \infty}\sqrt{\pi n}\PP(S_{2n}=0) \geq 1 ;
  $$
  \item  if $p\geq 1/2$, then $\EE \zeta_m \leq m^2$ and 
  $$
  \limsup_{n\to \infty} \frac{|S_n|}{ \sqrt{2n \log \log n}} \geq 1, \ a.s. \quad \limsup_{n\to \infty}\sqrt{\pi n}\PP(S_{2n}=0) \leq 1.
   $$
      \end{enumerate} 
\end{corollary}
\begin{remark}
  The estimates on the expected exit times of the MERW will be given in Proposition \ref{estexpectMERWhit}.
\end{remark}
\begin{proof}[Proof of Corollary \ref{erwsrwcouple}]
  (i) It is known that for a simple random walk $\tilde{S}$ with $\tilde{S}_1=1$, i.e. an ERW with parameter $1/2$, if $\tilde{\zeta_m}:=\inf\{n\in \NN: |\tilde{S}_n|=m\}$, then $\EE \tilde{\zeta}_m=m^2$ (gambler’s ruin problem, see e.g. Section 14.4, Chapter 10 \cite{gut2013probability}). Moreover, $\tilde{S}$ is recurrent and satisfies the law of the iterated logarithm and 
  $$
 \PP(\tilde{S}_{2n}=0) = \frac{C_{2n}^n}{2^{2n}} \sim \frac{1}{\sqrt{\pi n}}
  $$
  by Stirling’s formula. By Proposition \ref{monotonicitythm}, we can define $S$ and $\tilde{S}$ on the same probability space such that 
  \begin{equation}
    \label{Sn12SRWless}
    |S_n| \leq |\tilde{S}_n|, \quad \forall n\in \NN 
  \end{equation}
 In particular, $S$ visits 0 infinitely often a.s., and thus $S$ is recurrent by Proposition \ref{merw01} (see Section \ref{Lyasuintro}). In addition, (\ref{Sn12SRWless}) implies that 
 $$\zeta_m\geq \tilde{\zeta}_m, \quad \PP(S_{2n}=0) \geq \PP(\tilde{S}_{2n}=0)$$ whence (i) follows. (ii) is proved similarly.
\end{proof}

Similarly, in dimensions $d\geq 2$, one may expect that the MERW tends to go further from the origin if $p$ becomes larger, as shown in Figure \ref{simuMERW2}. However, the MERW is not Markovian in general, which makes it difficult to couple MERWs with different parameters as in Proposition \ref{monotonicitythm}. Note that the ERW ($d=1$) is a time-inhomogeneous Markov chain, which plays a crucial role in our proof of Proposition \ref{monotonicitythm}.

For the purpose of the proof, we introduce a new type of random walk on $\ZZ^d$ which we call the d-elephant random walk. A d-elephant random walk (d-ERW) $(\tilde{S}_n)_{n\in \NN}$ with a pair of parameters $(p,q)\in [0,1]^2$ is a nearest-neighbor random walk on $\ZZ^d$ defined as below.

First let $\tilde{S}_0=0 \in \ZZ^d$ and assume that $\tilde{S}_1=\tilde{\sigma}^{(q)}_1=e_1$ unless otherwise specified. As in (\ref{numiaxis}), at any time $n \geq 1$, for $i\in \{1,2,\cdots,d\}$, we let
\begin{equation}
  \label{tildebindef}
  \tilde{b}_n^{(q)}(i):= \#\{1\leq j \leq n: \tilde{\sigma}^{(q)}_j\in \{e_i,-e_i\}\}
\end{equation}
be the number of steps of $\tilde{S}$ along the i-th axis up to time $n$. As in (\ref{defcin}), we write 
\begin{equation}
  \label{tildecndef}
  \tilde{c}_n^{(q)}(i):=\frac{2dp-1}{2d-1} \frac{\tilde{b}_n^{(q)}(i)}{n}+\frac{2-2p}{2d-1} , \quad i=1,2,\cdots, d.
\end{equation}
Note that $\sum_{i=1}^d \tilde{c}_n^{(q)}(i)=1$. Given $\tilde{S}_0,\tilde{S}_1,\cdots,\tilde{S}_n=\tilde{x}$, we then define a random variable $\tilde{\sigma}^{(q)}_{n+1}$ whose (conditional) law is given by 
\begin{equation}
  \label{dERWdefequ}
  \PP(\tilde{\sigma}^{(q)}_{n+1}= \pm e_i) = \tilde{c}_n^{(q)}(i)(\frac{1}{2}\pm \frac{(2q-1)\tilde{x}(i)}{2\tilde{b}_n^{(q)}(i)}\mathds{1}_{\tilde{b}_n^{(q)}(i)\neq 0}), \quad i=1,2,\cdots,d,
\end{equation}
where we recall that $\tilde{x}(i)$ denotes the $i$-th coordinate of $\tilde{x}$. We then set
$\tilde{S}_{n+1}:=\tilde{S}_n+\tilde{\sigma}^{(q)}_{n+1}$. The definition of the d-ERW is motivated by its connection to the MERW.

\begin{proposition}
  \label{coupleMERW}
  Let $(p,q_1,q_2)\in [0,1]^3$. We can define nearest-neighbor random walks $S$, $\tilde{S}^{(q_1)}$ and $\tilde{S}^{(q_2)}$ on $\ZZ^d$ on the same probability space such that 
  \begin{enumerate}[topsep=0pt, partopsep=0pt,leftmargin=5pt, align=left,  label=(\roman*)]
\item $S$ is a MERW with parameter $p$, $\tilde{S}^{(q_1)}$ and $\tilde{S}^{(q_2)}$ are d-ERWs with pairs of parameters $(p,q_1)$ and $(p,q_2)$, respectively.
\item Let $\sigma_n:=S_n-S_{n-1}$ and $\sigma_n^{(q_k)}:=\tilde{S}^{(q_k)}_n-\tilde{S}^{(q_k)}_{n-1}$, $k=1,2$. Then, for any $n\geq 1$ and $i\in \{1,2,\cdots,d\}$, 
$$\sigma_n \in  \{e_i,-e_i\}\Leftrightarrow \tilde{\sigma}_n^{(q_1)} \in  \{e_i,-e_i\} \Leftrightarrow \tilde{\sigma}_n^{(q_2)} \in  \{e_i,-e_i\}$$ 
 Or equivalently, $ b_n(i)= \tilde{b}^{(q_1)}_n(i)=\tilde{b}^{(q_2)}_n(i)$ where $b_n(i)$ is defined in (\ref{numiaxis}), and $\tilde{b}^{(q_1)}_n(i)$, $\tilde{b}^{(q_2)}_n(i)$ are defined as in (\ref{tildebindef}) with $q=q_1,q_2$, respectively.
\item For any $n\geq 1$ and $k\in \{1,2\}$, conditional on $\{b_m(i)\}_{1\leq i \leq d, 1\leq m \leq n}=\{a_m(i)\}_{1\leq i \leq d, 1\leq m \leq n}$, the distribution of $(\tilde{S}_m^{(q_k)})_{1\leq m \leq n}$ is given by
\begin{equation}
  \label{condindderw}
  (\tilde{S}_m^{(q_k)})_{1\leq m \leq n}\stackrel{\mathcal{L}}{=} \left((X^{(1)}_{a_m(1)},X^{(2)}_{a_m(2)},\cdots, X^{(d)}_{a_m(d)})\right)_{1\leq m \leq n}
\end{equation}
where $X^{(i)}$ $(1\leq i\leq d)$ are d independent one-dimensional ERWs with common parameter $q_k$, and $X^{(1)}_1=1$, $X^{(i)}_1 \sim \operatorname{Rademacher}(\frac{1}{2})$ for $i \neq 1$. 
\item If $p \leq 1 /(2 d)$, $q_1 \in\left[0, (2d-1)p/(2 dp-2p+1)\right]$, $q_2=1/2$, then 
\begin{equation}
  \label{SitildeSibound}
  |\tilde{S}^{(q_1)}_n(i)| \leq |S_n(i)| \leq |\tilde{S}^{(q_2)}_n(i)|, \quad \forall i \in 1,2,\cdots,d; \ n\in \NN
\end{equation}
If $p \geq 1 / (2 d)$, $q_1=1/2$ and $q_2\in \left[ (2d-1)p/(2 dp-2p+1),1\right]$, then (\ref{SitildeSibound}) is still true. 
\item If $q_1\leq q_2$, then 
$$|\tilde{S}_n^{(q_1)}(i)| \leq  |\tilde{S}_n^{(q_2)}(i)|, \quad \forall i \in 1,2,\cdots,d;\ n\in \NN$$   
\end{enumerate} 
\end{proposition}
\begin{remark}
  \label{coupleremq1q2}
  (i) As will be shown in its proof, Proposition \ref{coupleMERW} enables the coupling for an arbitrary number of parameters $q_1,q_2,\cdots$ such that $\tilde{S}^{(q_1)}$, $\tilde{S}^{(q_2)}, \cdots$ are d-ERWs with pairs of parameters $(p,q_1)$, $(p,q_2),\cdots$  \\ 
 (ii) By Proposition \ref{coupleMERW} (iii), given $(b_n(i))_{1\leq i \leq d,n \geq 1}$, the coordinates indexed by their respective moves of a d-ERW behave like $d$ independent ERWs and $(b_n(i))_{1\leq i \leq d,n \geq 1}$ decide which of those ERWs moves in each step. This justifies its name d-ERW. \\
 (iii) If $d=1$, then $b_n(1)=n$ and thus, by (\ref{tildecndef}) and (\ref{dERWdefequ}), a 1-ERW with a pair of parameters $(p,q)$ is simply an ERW with parameter $q$. Therefore, Proposition \ref{coupleMERW} (v) generalizes the coupling in Proposition \ref{monotonicitythm}.    
\end{remark}

If $S$ is a MERW with parameter $p$, then, by Proposition \ref{coupleMERW} (iv) and (v), for any $p \in [0,1]$,
\begin{equation}
  \label{MERWlowbdp0}
  |S_n(i)| \geq |\tilde{S}^{(0)}_n(i)|, \quad\forall i \in 1,2,\cdots,d; \ n\in \NN
\end{equation}
where $\tilde{S}^{(0)}$ is a d-ERW with parameters $p$ and $q=0$. By Proposition \ref{coupleMERW} (iii), we reduce the problem of finding a lower bound of the norm of the MERW to estimating the norms of ERWs.

Note that, for an ERW $(S_n)$, a central limit theorem for the position $S_n$, properly normalized, were established in the diffusive regime $p<3 / 4$ and the critical regime $p=3 / 4$, see \cite{baur2016elephant}, \cite{bercu2017martingale} \cite{coletti2017central}, \cite{coletti2017strong}. The Berry-Esseen type bounds in \cite{fan2021cramer} and \cite{fan2022cram} give estimations for the rate of convergence to the normal distribution. We prove an improved version by using a recent result by  Dedecker, Fan, Hu and Merlev\`ede \cite{MR4646949}.

\begin{theorem}[Berry-Esseen type bounds]
  \label{cltp}
  Let $(S_n)_{n\geq 0}$ be an elephant random walk with parameter $p\leq 3/4$. 
  \begin{enumerate}[topsep=0pt, partopsep=0pt,leftmargin=5pt, align=left,  label=(\roman*)]
 \item If $p<3/4$, then there exists a constant $C(p)$ such that for all $n \geq 1$,
  $$
      \sup_{t\in \RR}|\PP(\frac{\sqrt{3-4p}S_n}{\sqrt{n}} \leq t)-\Phi(t)| \leq \frac{C(p) }{n^{\min(1,3-4p)/2}}
  $$
where $\Phi$ is the distribution function of the standard normal distribution. 
\item  If $p=3/4$, then there exists a constant $C$ such that for all $n > 1$,
$$
  \sup_{t\in \RR}|\PP(\frac{S_n}{\sqrt{n\log n}} \leq t)-\Phi(t)| \leq \frac{C }{\sqrt{\log n}}.
$$
\end{enumerate} 
\end{theorem}

If $S$ is an ERW with parameter $p=0$, then by Theorem \ref{cltp} (i), for any $\nu <1/2$, there exists a positive constant $C(\nu)$ such that for any $n\geq 1$,
\begin{equation}
  \label{ERWnuprobbd}
  \PP(|S_n|\leq n^{\nu} ) = \PP(\frac{\sqrt{3}|S_n|}{\sqrt{n}} \leq \frac{\sqrt{3}n^{\nu} }{\sqrt{n}}) \leq \Phi(\frac{\sqrt{3}n^{\nu} }{\sqrt{n}})-\Phi(-\frac{\sqrt{3}n^{\nu} }{\sqrt{n}}) + \frac{2C(0)}{\sqrt{n}}  \leq \frac{C(\nu)}{n^{\frac{1}{2}-\nu}}
\end{equation}
where we used the Lipschitz property of $\Phi$. Let $\varepsilon\in (0,1/(2d))$ be a constant and $\tilde{S}$ be a d-ERW ($d\geq 3$) with a pair of parameters $(p,0)$ as in Proposition \ref{coupleMERW}.  By (\ref{ERWnuprobbd}) and Proposition \ref{coupleMERW} (iii), we can show that for some positive constant $C(\varepsilon,\nu,d)$, 
$$
\PP(\|\tilde{S}_n\| \leq n^{\nu}) \leq \frac{C(\varepsilon,\nu,d)}{n^{d(\frac{1}{2}-\nu)}} +\PP(\bigcup_{i=1}^d|\frac{b_n(i)}{n}-\frac{1}{d}|>\varepsilon), \quad \forall n\geq 1
$$
For $p\neq 1$, Lemma \ref{ratecon1overd} enables us to bound the second term on the right-hand side and show that $\PP(\|\tilde{S}_n\| \leq n^{\nu})$ is summable if $\nu <\frac{1}{2}-\frac{1}{d}$. We can then prove the following result by (\ref{MERWlowbdp0}).

\begin{proposition}
  \label{MERWtran12dcouple}
 For $d\geq 3$, $p\in [0,1)$ and $q_1=0$, let $S$ and $\tilde{S}:=\tilde{S}^{(q_1)}$ be respectively a MERW and a d-ERW as in Proposition \ref{coupleMERW}. Then, for any $\nu \in (0,\frac{1}{2}-\frac{1}{d})$,
  $$
  \EE( \sum_{n=1}^{\infty}\mathds{1}_{\{\|S_n\|\leq n^{\nu}\}}) \leq \EE( \sum_{n=1}^{\infty}\mathds{1}_{\{\|\tilde{S}_n\|\leq n^{\nu}\}}) < \infty
  $$
in particular, almost surely, $\|S_n\|> n^{\nu}$ for all large $n$. 
  \end{proposition}

  \subsubsection{Connection to P\'{o}lya-type urns and continuous-time embedding}

  A MERW $(S_n)$ on $\ZZ^d$ can be embedded into a continuous-time branching process. We first model the MERW with the help of P\'{o}lya-type urns. This technique, among others, enables us to prove Proposition \ref{Lnot0MERW}. This connection to P\'{o}lya-type urns was observed by Baur and Bertoin \cite{baur2016elephant}.
  
  Suppose we have an urn of balls of $2d$ directions (or colors). The set of directions is given by $\{e_1,-e_1,\cdots,e_d,-e_d\}$.  The composition of the urn at any time $n\geq 1$ is specified by a $2 d$-dimensional random vector $N_n=(N_n(1),N_{n}(-1),N_n(2),N_{n}(-2),\cdots,N_n(d),N_{n}(-d))$. We assume that the urn has initial composition $N_1=(1,0,\cdots,0)$. At any time $n\geq 1$, we draw a ball uniformly at random from the urn, observe its direction, put it back to the urn and add with probability $p$ a ball of the same direction, or add a ball of the $2 d-1$ remaining directions each with probability $(1-p) /(2 d-1)$. We update $N_{n+1}$ accordingly. Then, $(S_n)_{n \geq 1}$ defined by 
  $$
  S_n:=\sum_{i=1}^d (N_n(i)-N_{n}(-i))e_i, \quad n \geq 1 
  $$
  is a MERW on $\ZZ^d$ and so there is no ambiguity when we use the same notation $N_n(\pm i)$ as in (\ref{defNni}). 
  
  \begin{remark}
    The MERW $S$ with parameter $p=1$ defined in Section \ref{secdefmainre} is trivial: $S_n=ne_1$. In general, we can start a MERW from time $m \in \NN\backslash \{0\}$ with initial conditions $N_m(\pm i):=\alpha_{\pm i}$ $(i=1,2,\cdots,d)$ where $\alpha_{\pm i}\in \NN$ and $\sum_{i=1}^d (\alpha_{i}+\alpha_{-i})=m$. Let $I:=\{j\in \ZZ: \alpha_j>0\}$. Then, the case $p=1$ corresponds to an $|I|$-color P\'{o}lya urn with initial composition $\alpha=(\alpha_{j})_{j\in I}$. Note that $N_n(j)=0$ for any $n\geq m$ if $j\notin I$. A basic result is that, see e.g. \cite{MR1447347},
    $$
 \lim_{n\to \infty}\frac{(N_n(j))_{j\in I}}{n}\ \text{exists with limiting distribution} \ \operatorname{Dir}(\alpha)
    $$
   whence Proposition \ref{Lnot0MERW} follows. There is thus no loss of generality in assuming that $m=1$.
  \end{remark}
  
  Now we use the continuous-time embedding by Athreya and Karlin \cite{athreya1968embedding} and Athreya and Ney (Section 9, Chapter 5 \cite{athreya2004branching}). We define the continuous-time process $(U_t)_{t\geq 0} \in \RR_+^{2d}$ as follows: Set $U_0=(1,0,\cdots,0)$. Suppose we have an urn of balls of $2d$ directions $\{e_1,-e_1,\cdots,e_d,-e_d\}$  with initial composition $(1,0,\cdots,0)$. In the urn, at any moment, each ball is equipped with an Exp(1)-distributed random clock, all the clocks being independent. When the clock of a ball of a direction, say $e_i$, rings, we add a ball of $e_i$ to the urn with probability $p$ or add a ball of the $2 d-1$ remaining directions each with probability $(1-p) /(2 d-1)$. Note that by our assumption, a new clock is launched simultaneously to replace the clock that rings, and each ball added is equipped with its own clock. Let $U_{t}(\pm i)$ be the numbers of balls of $\pm e_i$ in the urn at time $t\geq 0$, respectively. Let 
  \begin{equation}
    \label{defUt}
    U_t:=(U_{t}(1),U_{t}(-1),\cdots,U_{t}(d),U_{t}(-d)),\quad t\geq 0
  \end{equation}
  The successive jump times of $\left(U_t, t \geq 0\right)$, will be denoted by $0=\tau_0<\tau_1<\tau_2<\cdots$. By the memoryless property of exponentials, we easily see that $(U_{\tau_n},n\geq 0)$ is equal in law to $(N_{n+1},n\geq 0)$. We may define them on the same probability space such that $(U_{\tau_n},n \geq 0)=(N_{n+1},n\geq 0)$. Note that $(\tau_{n+1}-\tau_n)_{n\in \NN}$ are independent and exponentially distributed, with $\tau_{n+1}-\tau_n$ having parameter $n+1$. Moreover, the processes $\left(\tau_n\right)_{n \in \NN}$ and $\left(U_{\tau_n}\right)_{n \in \NN}$ are independent. Note that this continuous-time embedding (in the form of a point process) was used by Bertoin \cite{bertoin2021universality} to obtain a version of Donsker’s invariance principle for the step-reinforced random walk in Remark \ref{remdefMERW} (iii).
    
    Note that when $p\in (p_d,1)$, $e_1-e_{2}, e_3-e_{4},\cdots,e_{2d-1}-e_{2d} \in \RR^{2d}$ are orthogonal eigenvectors of the mean replacement matrix
    $$
\left(\begin{array}{ccccc}
p & \frac{1-p}{2 d-1} & \cdots & \cdots & \frac{1-p}{2 d-1} \\
\frac{1-p}{2 d-1} & \ddots & \ddots & & \vdots \\
\vdots & \ddots & \ddots & \ddots & \vdots \\
\vdots & & \ddots & \ddots & \frac{1-p}{2 d-1} \\
\frac{1-p}{2 d-1} & \cdots & \cdots & \frac{1-p}{2 d-1} & p
\end{array}\right)
$$
    corresponding to the eigenvalue $a=(2 d p-1) /(2 d-1)\in (1/2,1)$. By results in Chapter 5 \cite{athreya2004branching}, see also Theorem 3.1 \cite{janson2004functional}, the limits
    \begin{equation}
      \label{limMERWxiW}
      \xi:=\lim _{t \rightarrow \infty} \frac{\sum_{i=1}^{d}U_t(i)+U_t(-i)}{e^t}=\lim _{n \rightarrow \infty} \frac{n}{e^{\tau_n}}, \ W_d:=\lim _{t \rightarrow \infty} \frac{(U_{t}(i)-U_{t}(-i))_{1\leq i \leq d}}{e^{a t}}
    \end{equation}
  exist a.s.. In addition, $\EE W_d = e_1\in \RR^d$ by Theorem 2, Section 8, Chapter 5 \cite{athreya2004branching}, and $\xi$ has an Exp(1)-distribution by a result of D. Kendall \cite{kendall1966branching}. Since $S_{n+1}(i)=U_{\tau_n}(i)-U_{\tau_n}(-i)$, by (\ref{superlimY}),
  \begin{equation}
    \label{WYxi}
      W_d=Y_d\xi^{a}, \quad a.s.
  \end{equation} 
  Notice that $Y_d$ and $\xi$ are independent by the independence of $\left(\tau_n\right)_{n \in \NN}$ and $\left(U_{\tau_n}\right)_{n \in \NN}$.
  We can then get some information on the distribution of $Y_d$ by investigating the distribution of $W_d$.
  
  Adapting the method developed by Chauvin, Pouyanne and Sahnou in the proof of Proposition 4.2 \cite{chauvin2011limit} and Janson in Theorem 3.9 \cite{janson2004functional}, we can prove in Proposition \ref{fgodepropw} that the characteristic function $\varphi$ of $w:=\sum_{i=1}^d W_d(i)$ satisfies the following ODE
  \begin{equation}
    \label{wcharode}
    \varphi(x)+ a x \varphi^{\prime}(x)=\frac{dp+d-1}{2d-1} \varphi^2(x)+\frac{d-dp}{2d-1}|\varphi(x)|^2, \quad x\in \RR
  \end{equation}
  We will study in Section \ref{secLneq0} the ODEs satisfied by the real and imaginary parts of $\varphi$ and prove in Proposition \ref{wdensity} that 
  $$\sup_{x\in \RR}|x^{ \frac{1}{a}}\varphi(x)|<\infty $$ 
  Thus, $|\varphi| \in L^1$ and $w$ admits a density, which completes the proof of Proposition \ref{Lnot0MERW}.

  For the ERW, (\ref{wcharode}) enables us to obtain the moments $(\EE Y_1^n)_{n \geq 1}$ of $Y_1$ by establishing recursive equations for $(\EE Y_1^n)_{n \geq 1}$, see Corollary \ref{momentWYrec}. We will show in Proposition \ref{momcharWY1} that $(\EE Y_1^n)_{n \geq 1}$ characterize the law of $Y_1$.

  \subsubsection{Lyapunov functions method}
  \label{Lyasuintro}

  The main technique we use to prove Theorem \ref{phasetranZ2MERW} is the Lyapunov functions method developed by Lamperti \cite{lamperti1960criteria}. A good introduction to the Lyapunov functions method can be found in a book by Menshikov, Popov and Wade \cite{menshikov2016non}. 

  For an irreducible Markov random walk $(X_n)_{n\in \NN}$ on $\ZZ^d$, one may prove the recurrence property if one can find a function $f: \ZZ^d \rightarrow \mathbb{R}_{+}$ and a ball $B(0,r)$ ($r>0$) such that
\begin{equation}
  \label{LyapfdefcondE}
  \mathbb{E}\left[f\left(X_{n+1}\right)-f\left(X_n\right) \mid \sigma(X_m: m\leq n) \right] \leq  0, \text { if } X_n \notin B(0,r)
\end{equation}
and $f(x) \rightarrow \infty$ as $|x| \rightarrow \infty$. Then we can deduce that $X$ returns to $B(0,r)$ infinitely often a.s.. Indeed, if $X_k \notin B(0,r)$ for some $k\in \NN$, let $\tau:=\inf\{j\geq k: X_j \in B(0,r)\}$. Then by supermartingale convergence theorem, $\{f(X_{n\wedge \tau})\}_{n\geq k}$ is a.s. convergent. However, $X_n$ goes to infinity a.s. on $\{\tau=\infty\}$ and thus $f(X_{n\wedge \tau}) \to \infty$ a.s. on $\{\tau=\infty\}$ whence $\PP(\tau=\infty)=0$. Since $k$ is arbitrary, $(X_n)_{n\in \NN}$ is recurrent. 
\begin{remark}
  If we further assume that for some $\varepsilon>0$,
  $$
  \mathbb{E}\left[f\left(X_{n+1}\right)-f\left(X_n\right) \mid \sigma(X_m: m\leq n) \right] \leq  -\varepsilon, \text { if } X_n \notin B(0,r)
  $$
 which is referred to as Foster’s criterion \cite{MR0056232}, then one can show that $X$ is positive recurrent.
\end{remark}
Using the language of ODEs, we may say that $f$ in (\ref{LyapfdefcondE}) is a Lyapunov function and the ball $B(0,r)$ is an attractor. However, the MERW on $\ZZ^2$ is not Markovian in general and the region of attraction grows as time grows. Thus, in our case, we need to adapt the classical Lyapunov functions method.

  The following proposition implies that, like a Markov random walk, a MERW will go to infinity if it never returns to a finite ball $B(0,r)$. Moreover, it is recurrent if it returns to $B(0,r)$ infinitely often almost surely.

\begin{proposition}
  \label{merw01}
  Let $S$ be a MERW on $\ZZ^d$. Denote $A_1=\{$every vertex is visited by S infinitely often$\}$ and $A_2=\{$every vertex is visited by S finitely often$\}$. Then, 
  $$
\PP(A_1\cup A_2) = 1
  $$
\end{proposition}
\begin{proof}
  The case $p=1$ is trivial. We may assume that $p\in [0,1)$.
  Let $x\in \ZZ^d$ and let $y$ be a neighbor of $x$, say $y=x+e_i$ for some $i\in \{1,2,\cdots,d\}$. By the conditional Borel-Cantelli lemma, see e.g. \cite{dacunhaprobability}, 
  $$
  \{\sum_{n=1}^{\infty} \mathds{1}_{\{S_n=x,S_{n+1}=y\}} =\infty\} =  \{\sum_{n=1}^{\infty} \EE(\mathds{1}_{\{S_n=x,S_{n+1}=y\}}|\FF_n)=\infty\}, \quad a.s.
  $$
  Recall $N_n(\pm i)$ from (\ref{defNni}). Since $N_n(i)+N_{n}(-i) \leq n$, $N_n(i)-N_{n}(-i)=S_n(i)$, 
  $$
 \frac{N_n(i)}{n} \leq \frac{1}{2}+ \frac{S_n(i)}{2n}
  $$
  Thus, if $2dp-1<0$, by (\ref{Psigman1}),
  $$
  \EE(\mathds{1}_{\{S_n=x,S_{n+1}=y\}}|\FF_n) \geq \left(\frac{2dp-1}{2d-1}(\frac{1}{2}+ \frac{x(i)}{2n})+ \frac{1-p}{2d-1}\right)\mathds{1}_{\{S_n=x\}}
  $$
  Since $x$ is fixed, for any $n \geq 3|x(i)|$, the right hand side is lower bounded by $(p+\frac{1-2dp}{6d-3})\mathds{1}_{\{S_n=x\}}$. If $2dp-1\geq 0$, 
  $$
  \EE(\mathds{1}_{\{S_n=x,S_{n+1}=y\}}|\FF_n) \geq \frac{1-p}{2d-1}\mathds{1}_{\{S_n=x\}}
  $$
In either case, $ \sum_{n=1}^{\infty} \mathds{1}_{\{S_n=x,S_{n+1}=y\}} =\infty$ a.s. on $E_x:=\{x \ \text{is visited infinitely many times}\}$. Indeed, this proves that a.s. on $E_x$, all neighbors of $x$ are visited infinitely many times. By the connectedness of $\ZZ^d$, this implies that for any $x, y\in \ZZ^d$, a.s. on the $E_x$, $y$ is visited infinitely many times. Thus, either all vertices are visited infinitely many times or all vertices are visited finitely often.
 \end{proof}

In dimension $d=2$, for $p <5/8$, we may find a proper Lyapunov function $f$ such that for $n$ less than some stopping time, we can give an upper bound on $\mathbb{E}\left[f\left(S_{n+1}\right)-f\left(S_n\right) \mid \FF_n \right]$, see e.g. Inequalities (\ref{2dEfdiff}) and (\ref{upbdsumajS}) in the proof Proposition \ref{Snsio}. Using these estimates, we can first prove a weaker result. Recall $a$ in (\ref{superlimY}). In the two-dimensional case, $a=(4p-1)/3$.
\begin{proposition}
  \label{Snsio}
  Let $S$ be a MERW on $\ZZ^2$ with parameter $p < 5/8$, i.e. $a<1/2$, let $s\in (a,1/2)$ be a positive constant. Then, almost surely, $\|S_n\|\leq n^s$ infinitely often.
\end{proposition}

Starting from time $n$ with $\|S_n\|\leq n^s$, the attraction remains strong enough before $\theta:=\inf\{j \geq n: \|S_j\| \geq j^t \}$ where $t\in (s,1/2)$. We can show by martingale techniques that with positive probability uniformly bounded away from 0, $S$ will return to $B(0,r)$ before $\theta$. We then conclude by using Levy's 0-1 law and Proposition \ref{merw01} that $S$ is recurrent for $p<5/8$. Combined with Proposition \ref{Lnot0MERW} and Proposition \ref{z258logn}, this completes the proof of Theorem \ref{phasetranZ2MERW}.

\subsubsection{Organization of the paper} We describe in detail the coupling technique in Section \ref{seccouplMERW}. Proposition \ref{monotonicitythm}, Proposition \ref{coupleMERW} and Proposition \ref{MERWtran12dcouple} are proved in Section \ref{coupleERWs}, Section \ref{coupleMERWderw} and Section \ref{proofnu1211d}, respectively.

Using the continuous-time embedding, we prove the ODE (\ref{wcharode}) and Proposition \ref{Lnot0MERW} in Section \ref{secLneq0}. Moreover, Corollary \ref{angularasym} is proved. For $Y_1$ defined in (\ref{superlimY}), we investigate its distribution in Section \ref{dismomenY1}.

Section \ref{2dERWsec} is devoted to investigating the recurrence and transience properties of the MERW on $\ZZ^2$ with $p\leq 5/8$.
In Section \ref{proofweakerrec}, we use the Lyapunov functions method to prove Proposition \ref{Snsio}. The proofs of Proposition \ref{z258logn} and Theorem \ref{phasetranZ2MERW} are presented in Section \ref{proof2d}.

Some other results on the MERW are proved in Section \ref{someotherre}. In Section \ref{estexpexit}, we study the asymptotic behavior of $\EE \|S_n\|^2$ and provide some estimates of the expected exit times of the MERW. Section \ref{secimprorate} is devoted to the proof of Proposition \ref{rateescapMERW2d}. In Section \ref{rateconclteswp0}, we prove Theorem \ref{cltp}.

\section{Proofs of the main results}

\subsection{Coupling}
\label{seccouplMERW}

\subsubsection{Coupling of ERWs: Proof of Proposition \ref{monotonicitythm}}
\label{coupleERWs}

\begin{proof}[Proof of Proposition \ref{monotonicitythm}]
  Let $\{U_n: n\geq 2\}$ be independent uniform random variables on $(0,1)$. Let $S_0=\tilde{S}_0=0$ and $S_1=\tilde{S}_1=1$. At time $n\geq 1$, suppose that $S_n=x\neq 0$, resp. $\tilde{S}_n=\tilde{x}\neq 0$, we let $S_{n+1}$ be a neighbor of $x$ such that $|S_{n+1}|=|x|+1$, resp. $\tilde{S}_{n+1}$ be a neighbor of $\tilde{x}$ such that $|\tilde{S}_{n+1}|=|\tilde{x}|+1$, if
  \begin{equation}
  \label{coupleUgofar}
  U_{n+1}<\frac{1}{2}+\frac{(2p_1-1)|x|}{2n}, \quad \text{resp.} \quad U_{n+1}<\frac{1}{2}+\frac{(2p_2-1)|\tilde{x}|}{2n}
  \end{equation}
  otherwise, let $S_{n+1}$ be such that $|S_{n+1}|=|x|-1$, resp. $\tilde{S}_{n+1}$ be such that $|\tilde{S}_{n+1}|=|\tilde{x}|-1$. If $S_n=0$, resp. $\tilde{S}_n=0$, set $S_{n+1}=1$, resp. $\tilde{S}_{n+1}=1$, if $U_{n+1}<1/2$, and otherwise, set $S_{n+1}=-1$, resp. $\tilde{S}_{n+1}=-1$.

  We claim that $S$ and $\tilde{S}$ defined above are ERWs with parameters $p_1$ and $p_2$, respectively. 
  
  \textbf{Proof of the claim}: As in (\ref{defNni}), we denote by $N_n(\pm i)$ the number of steps of $S$ in the directions $\pm e_1$ up to time $n$. Given $S_n=x$, we have $N_n(1)=(n+x)/2$ and $N_n(-1)=(n-x)/2$. Then, by definition, 
  $$\PP(S_{n+1}=x+1|S_n=x,S_0,S_1,\cdots,S_{n-1}) =\frac{1}{2}+\frac{(2p_1-1)x}{2n}=p_1\frac{N_n(1)}{n}+(1-p_1)\frac{N_n(-1)}{n}$$ 
  By the definition of ERW or (\ref{Psigman1}), $S$ is an ERW with parameter $p_1$. The proof of that $\tilde{S}$ is an ERW with parameter $p_2$ is similar. The claim is proved.
  
  It remains to prove (ii).  We proceed by induction on $n$. Trivally, it holds for $n \leq 1$. Assume that the statement holds for all $n\leq k$ for some $k\geq 1$. Note that $|\tilde{S}_k|-|S_k|$ is even. If $|\tilde{S}_k|-|S_k| \geq 2$, then the statement holds for $n=k+1$. If $|\tilde{S}_k|=|S_k|$, then 
  $$\frac{1}{2}+\frac{(2p_1-1)|S_k|}{2n} \leq \frac{1}{2}+\frac{(2p_2-1)|\tilde{S}_k|}{2n}$$
  thus, $|\tilde{S}_{k+1}|=|\tilde{S}_{k}|+1$ if $|S_{k+1}|=|S_k|+1$. Again, this implies that (ii) holds for $n=k+1$.
  \end{proof}

\subsubsection{Coupling of a MERW and d-ERWs}
\label{coupleMERWderw}

\begin{proof}[Proof of Proposition \ref{coupleMERW}]
  Let $\{U_n\}_{n\geq 2}$ and $\{U_n^{(i)}\}_{n\geq 2, 1\leq i \leq d}$ be i.i.d. uniform random variables on $(0,1)$. Set $S_0=\tilde{S}_0^{(q_1)}=\tilde{S}_0^{(q_2)}=0$ and $S_1=\tilde{S}_1^{(q_1)}=\tilde{S}_1^{(q_2)}=e_1 \in \RR^d$. For $n\geq 1$ and $i\in \{1,2,\cdots,d\}$, let $b_n(i),c_n(i)$ be defined as in (\ref{numiaxis}), (\ref{defcin}), and set $c_n(0):=0$. Write $I_i:=(\sum_{j=0}^{i-1}c_n(j),\sum_{j=0}^{i}c_n(j))$. We define $\sigma_{n+1}$, $\tilde{\sigma}_{n+1}^{(q_k)}$ ($k=1,2$) as follows. Assume $U_{n+1}\in I_i$ for some $i \in \{1,2,\cdots, d\}$. For any $k \in \{1,2\}$, suppose that $S_{n}(i)\neq 0$, resp. $S_{n}^{(q_k)}(i)\neq 0$, then we set $\sigma_{n+1}=\operatorname{sgn}(S_n(i))e_i$, resp. $\tilde{\sigma}_{n+1}^{(q_k)}=\operatorname{sgn}(\tilde{S}_{n}^{(q_k)}(i))e_i$ if
\begin{equation}
  \label{UniMERWdERW}
  U_{n+1}^{(i)}<\frac{1}{2}+\frac{(2dp-1)|S_n(i)|}{2n(2d-1)c_n(i)}, \quad \text{resp.} \quad U_{n+1}^{(i)}<\frac{1}{2}+\frac{(2q_k-1)|\tilde{S}_{n}^{(q_k)}(i)|}{2b_n(i)}
\end{equation}
(note that $c_n(i)\neq 0$ if $S_n(i)\neq 0$ and $b_n(i)\neq 0$ if $\tilde{S}_{n}^{(q_k)}(i)\neq 0$), otherwise, set $\sigma_{n+1}=-\operatorname{sgn}(S_n(i))e_i$, resp. $\tilde{\sigma}_{n+1}^{(q_k)}=-\operatorname{sgn}(\tilde{S}_{n}^{(q_k)}(i))e_i$. Suppose that $S_n(i)=0$, resp. $S_{n+1}^{(q_k)}(i)=0$, then as in the proof Proposition \ref{monotonicitythm}, we set $\sigma_{n+1}=e_i$, resp. $\tilde{\sigma}_{n+1}^{(q_k)}=e_i$, if $U_{n+1}^{(i)}<1/2$, and otherwise, set $\sigma_{n+1}=-e_i$, resp. $\tilde{\sigma}_{n+1}^{(q_k)}=-e_i$.

Claim: $S$ and $\tilde{S}^{(q_k)}$ ($k=1,2$) defined above satisfy the requirements in Proposition \ref{coupleMERW}.

\textbf{Proof of the claim}: (i)  Given $S_0,S_1,\cdots,S_n$, by (\ref{UniMERWdERW}), the conditional probability that $S_{n+1}=S_n+ e_i$ is 
$$
\begin{aligned}
  &\quad c_n(i)(\frac{1}{2}+\frac{(2dp-1)S_n(i)}{2n(2d-1)c_n(i)}\mathds{1}_{c_n(i)\neq 0})=\frac{c_n(i)}{2}+ \frac{(2dp-1)S_n(i)}{2n(2d-1)} \\
  &= \frac{(2dp-1)(b_n(i)+S_n(i))}{2n(2d-1)} + \frac{1-p}{2d-1}=\frac{(2dp-1) N_{n}(i)}{(2d-1)n} + \frac{1-p}{2d-1}
\end{aligned}
$$
which agrees with the conditional probability in (\ref{Psigman1}). For the conditional probability that $S_{n+1}=S_n-e_i$, the proof is similar. Thus, $S$ is a MERW with parameter $p$. Similarly, one can show that the conditional law of $ \tilde{\sigma}_{n+1}^{(q_k)}$ agrees with  (\ref{dERWdefequ}), and thus $\tilde{S}^{(q_k)}$ is d-ERWs with a pair of parameters $(p,q_k)$.

(ii) Simply observe that for any $n\geq 2$, $U_n\in I_i \Leftrightarrow \sigma_n \in  \{e_i,-e_i\}\Leftrightarrow \tilde{\sigma}_n^{(q_k)} \in  \{e_i,-e_i\}$ ($k=1,2$).

(iii) For any $i\in \{1,2,\cdots,d\}$, suppose that the successive jump times of $b_{\cdot}(i)$ up to time $n$ are $1\leq k_1 <k_2<\cdots <k_{a_n(i)}$, i.e. for $1\leq \ell \leq a_n(i)$, 
$$
k_{\ell}:=\min\{m \geq 1: b_{m}(i)=\ell\}
$$
Note that the jump times are determined by $\{a_m(i)\}_{1\leq i \leq d, 1\leq m \leq n}$.
We define a random walk $X^{(i)}$ on $\ZZ$: set $X_0^{(i)}=0$ (and $X_1^{(i)}=1$ if $i=1$), and for any integer $\ell$ such that $1\leq \ell \leq a_n(i)$ if $i\neq 1$ (resp. $2\leq \ell \leq a_n(i)$ if $i=1$), let $X^{(i)}_{\ell}$ be a neighbor of $X^{(i)}_{\ell-1}\neq 0$ such that $| X^{(i)}_{\ell}|=|X^{(i)}_{\ell-1}|+1$ if
$$
U_{k_{\ell}}^{(i)}<\frac{1}{2}+\frac{(2q_1-1)|X^{(i)}_{\ell-1}|}{2(\ell -1)}
$$
and let $X^{(i)}_{\ell}$ be such that $|X^{(i)}_{\ell}|=|X^{(i)}_{\ell-1}|-1$ otherwise. As before, if $X^{(i)}_{\ell-1}=0$, we set $X^{(i)}_{\ell}=1$ if $U_{k_{\ell}}^{(i)}<1/2$ and set $X^{(i)}_{\ell}=-1$ otherwise. Then $(X^{(i)})_{1\leq i \leq d}$ are independent by the independence of $\{U_n^{(i)}: n\geq 2, i=1,2,\cdots, d\}$. From the proof of Proposition \ref{monotonicitythm}, see e.g. (\ref{coupleUgofar}), we see that $X^{(i)}$ is an ERW with parameter $q_1$. By the definition (\ref{UniMERWdERW}), $\tilde{S}_m^{(q_1)}=(X^{(i)}_{a_m(i)})_{1\leq i \leq d}$ for any $m\leq n$. The result for $\tilde{S}_m^{(q_2)}$ is proved similarly.

(iv) Assume that $p \leq 1 /(2 d)$, $q_1 \in\left[0, (2d-1)p/(2 dp-2p+1)\right]$, $q_2=1/2$. In particular, $c_n(i)\neq 0$ for all $n$ and $i$. We show that if $b_n(i)\neq 0$,
 \begin{equation}
  \label{coupleineSsqiq2}
  \frac{(2q_1-1)}{2b_n(i)} \leq \frac{(2dp-1)}{2n(2d-1)c_n(i)} \leq 0=(2q_2-1)
 \end{equation}
 The second inequality is obvious. The first inequality is trivial for $p=1/(2d)$. We may assume that $p<1/(2d)$. By the definition (\ref{defcin}), the first inequality in (\ref{coupleineSsqiq2}) is equivalent to 
 \begin{equation}
  \label{coupleineSsqiq2fbn}
  \frac{(1-2dp)b_n(i)}{n(2d-1)c_n(i)}= \frac{\frac{b_n(i)}{n}}{\frac{2-2p}{1-2dp}-\frac{b_n(i)}{n}} \leq 1-2q_1
 \end{equation}
 Observe that $(2-2p)/(1-2dp)>1$ and note that the function 
$$f(x):=\frac{x}{\frac{2-2p}{1-2dp}-x}, \quad x\in [0,1] $$
is increasing. Thus, $f(\frac{b_n(i)}{n}) \leq f(1) \leq 1-2q_1$ by the choice of $q_1$, which completes the proof of (\ref{coupleineSsqiq2}).
By (ii), $S(i)$, $\tilde{S}^{(q_1)}(i)$ and $\tilde{S}^{(q_2)}(i)$ move at the same times. Now (\ref{SitildeSibound}) is proved by induction as in the proof of Proposition \ref{monotonicitythm}. Indeed, if $|S_n(i)|=|\tilde{S}_n^{(q_1)}(i)|=x \neq 0$ and $U_{n+1} \in I_i$, then by (\ref{coupleineSsqiq2}), 
$$
\frac{1}{2}+\frac{(2q_1-1)|x|}{2b_n(i)} \leq \frac{1}{2}+\frac{(2dp-1)|x|}{2n(2d-1)c_n(i)}
$$
 Thus, $|S_{n+1}(i)|=|S_n(i)|+1$ if $|\tilde{S}_{n+1}^{(q_1)}(i)|=|\tilde{S}_{n}^{(q_1)}(i)|+1$. If $x=0$, then $|S_{n+1}(i)|=1=|\tilde{S}_{n+1}(i)|$. By induction, $|\tilde{S}^{(q_1)}_n(i)| \leq |S_n(i)|$ for all $i \in 1,2,\cdots,d$ and $n\in \NN$.
That $|S_n(i)| \leq |\tilde{S}^{(q_2)}_n(i)|$ is proved similarly.

 Now assume that $p \geq 1 / (2 d)$, $q_1=1/2$ and $q_2\in \left[ (2d-1)p/(2 dp-2p+1),1\right]$. We claim that if $b_n(i)\neq 0$ and $c_n(i)\neq 0$, then
  \begin{equation}
    \label{q1pq2ine2}
    2q_1-1= 0\leq \frac{(2dp-1)}{2n(2d-1)c_n(i)} \leq \frac{2q_2-1}{2b_n(i)}
  \end{equation}
  Again, we may assume $p>1/(2d)$. As in (\ref{coupleineSsqiq2fbn}), we need to show that
 $$
 -f(\frac{b_n(i)}{n}) \leq 2q_2-1
 $$
 which follows from the fact that $-f$ in increasing in $[0,1]$ and $-f(1) \leq 2q_2-1$ by the choice of $q_2$. The rest of the proof is similar.
 
 (v) The proof is the same as the proof of Proposition \ref{monotonicitythm}.
\end{proof}

 \subsubsection{Proof of Theorem \ref{MERWtran12d}}
 \label{proofnu1211d}

 Recall $b_n(i)$ from (\ref{numiaxis}). Theorem 3.21 \cite{janson2004functional} or Section V.9.3 \cite{athreya2004branching} implies that if $p\in [0,1)$ and $d\geq 2$, then almost surely, for any $i\in \{1,2,\cdots, d\}$,
 \begin{equation}
   \label{binlim1d}
   \lim_{n\to \infty}\frac{b_n(i)}{n}=\frac{1}{d}
 \end{equation}
To prove Theorem \ref{MERWtran12d}, we need some estimates on the rate of convergence in (\ref{binlim1d}). Let 
\begin{equation}
  \label{defetan}
  \eta_n(i):=\frac{b_n(i)}{n}-\frac{1}{d}, \quad n \geq 1
\end{equation}

\begin{lemma}
  \label{ratecon1overd}
  Let $p\in [0,1)$ and $d\geq 2$. For $r \geq 1$ and any $i \in \{1,2,\cdots,d\}$, there exists a constant $C(r,p,d)$ such that for any $n>1$,
  $$
  \EE |\eta_n(i)|^r\leq C(r,p,d)n^{-\frac{r}{2}},\ \text{if}\ p<p_d; \quad \EE |\eta_n(i)|^r\leq C(r,d)(\log n)^{\frac{r}{2}}n^{-\frac{r}{2}},\ \text{if}\ p=p_d
  $$
  and 
  $$
  \EE |\eta_n(i)|^r\leq C(r,p,d)n^{-\frac{2rd(1-p)}{2d-1}},\ \text{if} \ p>p_d
  $$
  Moreover, if $p\leq p_d$, for any $\nu <\frac{1}{2}$, $n^{\nu}\eta_n(i)$ converges to 0 a.s. 
\end{lemma}
\begin{proof}
  We prove only the case $i=1$, the proof of other cases is similar. For simplicity of notation, by a slight abuse of notation, we write $\eta_n:=\eta_n(1)$. By Lemma \ref{computationlem} (1), for any $n\geq 1$, we may write 
  $$
b_{n+1}(1):=b_n(1)+\xi_{n+1}
  $$
where $\xi_{n+1}$ is a Bernoulli random variable such that
$$
  1-\PP(\xi_{n+1}=0|\FF_n)=\PP(\xi_{n+1}=1|\FF_n)=c_n(1)=\frac{b_n(1)}{n}\frac{2dp-1}{2d-1}+\frac{2(1-p)}{2d-1}
$$
Thus, $(\eta_n)$ satisfies the following recursion 
  \begin{equation}
    \label{binrecur}
     \eta_{n+1}=\eta_n+\gamma_n(-\eta_n+\epsilon_{n+1}), \quad n\geq 1
  \end{equation}
  where 
  $$\gamma_n:=\frac{2d(1-p)}{(2d-1)(n+1)}, \quad \epsilon_{n+1}:=\frac{2d-1}{2d(1-p)} (\xi_{n+1}-\frac{b_n(1)}{n}\frac{2dp-1}{2d-1}-\frac{2(1-p)}{2d-1})$$
  Observe that $\EE(\epsilon_{n+1}|\FF_n)=0$ and $|\epsilon_{n+1}|<c_1(p,d)$ for some constant $c_1(p,d)$. Let $\beta_n:=\prod_{k=1}^{n-1}(1-\gamma_k)$ with the convention that $\beta_1=1$. Note that
  \begin{equation}
    \label{betanasy}
     \lim_{n\to \infty}\beta_n n^{\frac{2d(1-p)}{2d-1}} = c_2(p,d)
  \end{equation}
  where $c_2(p,d)$ is positive constant. By induction, one can show that 
  \begin{equation}
    \label{etanformula}
     \eta_n=\beta_n(\eta_1+ \sum_{j=1}^{n-1}\frac{\gamma_j}{\beta_{j+1}}\epsilon_{j+1}), \quad n\geq 1
  \end{equation}
  By (\ref{betanasy}) and Burkholder's inequality, see e.g. Theorem 2.10 \cite{hall1980martingale}, for $p<p_d$, there exist positive constants $C(r)$ and $C(r,p,d)$ such that for any $n>1$,
  \begin{equation}
    \label{rmsecondpart}
      \EE(\beta_n^r |\sum_{j=1}^{n-1}\frac{\gamma_j}{\beta_{j+1}}\epsilon_{j+1}|^r) \leq c_1^rC(r) \beta_n^r [\sum_{j=1}^{n-1} (\frac{\gamma_j}{\beta_{j+1}})^2]^{\frac{r}{2}} \leq C(r,p,d) n^{-\frac{r}{2}}
  \end{equation}
 Similarly, for any $n>1$,
  \begin{equation}
    \label{secondpartbdpd}
      \EE(\beta_n^r |\sum_{j=1}^{n-1}\frac{\gamma_j}{\beta_{j+1}}\epsilon_{j+1}|^r)\leq C(r,p,d) (\frac{\log n}{n})^{\frac{r}{2}}  \quad \text{if}\ p=p_d
  \end{equation}
  and 
$$
 \EE(\beta_n^r |\sum_{j=1}^{n-1}\frac{\gamma_j}{\beta_{j+1}}\epsilon_{j+1}|^r)\leq C(r,p,d) n^{-\frac{2rd(1-p)}{2d-1}}  \quad \text{if}\ p>p_d
$$
 Combined with (\ref{betanasy}) and (\ref{etanformula}), these inequalities complete the proof of the first assertion since $2d(1-p)/(2d-1) \geq 1/2$ if $p\leq p_d$.

Now we prove the last assertion. We first assume that $p< p_d$. For any $\nu <1/2$, choose $\kappa \in (\nu,1/2)$, then $\beta_n|\eta_1|<n^{-\kappa}/2$ for large $n$. Thus, by (\ref{rmsecondpart}) and Chebyshev's inequality,
\begin{equation}
  \label{etannnu}
  \PP(|\eta_n|\geq \frac{1}{n^{\kappa}}) \leq \PP( \beta_n|\sum_{j=1}^{n-1}\frac{\gamma_j}{\beta_{j+1}}\epsilon_{j+1}|\geq \frac{1}{2n^{\kappa}}) \leq \frac{2^r C(r,p,d) }{n^{r(\frac{1}{2}-\kappa)}}
\end{equation}
which is summable if we choose $r$ large enough such that $r(\frac{1}{2}-\kappa)>1$. By Borel-Cantelli lemma, almost surely, $n^{\kappa}|\eta_n| < 1$ for all but finite many $n$. In particular, $n^{\nu}|\eta_n| \to 0$. The case $p=p_d$ is proved similarly, where we use (\ref{secondpartbdpd}) instead of (\ref{rmsecondpart}).
\end{proof}

Now we are ready to prove Proposition \ref{MERWtran12dcouple} which implies Theorem \ref{MERWtran12d}(the case $p=1$ is trivial).

\begin{proof}[Proof of Proposition \ref{MERWtran12dcouple}]
 Fix $\varepsilon\in (0,\frac{1}{2d})$, define
$$
B_n(\varepsilon):=\{ x=(x_i)_{1\leq i \leq d} \in \NN^d: \sum_{i=1}^d x_i=n, |\frac{x_i}{n}-\frac{1}{d}|\leq \varepsilon, \ \forall i \in \{1,2,\cdots,d\}\}
$$
By Proposition \ref{coupleMERW} (iv), $\PP(\|S_n\|\leq n^{\nu})\leq \PP(\|\tilde{S}_n\|\leq n^{\nu} )$. Observe that
\begin{equation}
  \label{decompPtilSnnu}
  \PP( \|\tilde{S}_n\|\leq n^{\nu})\leq \PP(\{\|\tilde{S}_n\|\leq n^{\nu}\} \cap \{(b_n(i))_{1\leq i\leq d}\in B_n(\varepsilon)\})+\PP((b_n(i))_{1\leq i\leq d}\notin B_n(\varepsilon))
\end{equation}
It suffices to show that the right-hand side is summable. For the first term, we have
\begin{equation}
  \label{tildeSfirsttermbd}
     \begin{aligned}
   &\quad\ \PP(\{\|\tilde{S}_n\|\leq n^{\nu}\} \cap \{(b_n(i))_{1\leq i\leq d}\in B_n(\varepsilon)\}) \\
   &=  \sum_{x\in B_n(\varepsilon)}\PP(\|\tilde{S}_n\|\leq n^{\nu} |\bigcap_{i=1}^d\{b_n(i)=x_i\})\PP(\bigcap_{i=1}^d\{b_n(i)=x_i\}) \\
    &\leq \sum_{x\in B_n(\varepsilon)}\PP(\bigcap_{i=1}^d\{|\tilde{S}_n(i)|\leq n^{\nu}\}|\bigcap_{i=1}^d\{b_n(i)=x_i\}) \PP(\bigcap_{i=1}^d\{b_n(i)=x_i\}) 
  \end{aligned}
\end{equation}
  Conditional on $\bigcap_{i=1}^d\{b_n(i)=x_i\}$, by Proposition \ref{coupleMERW} (iii), $\tilde{S}_n$ is equal in law to $(Z^{(i)}_{x_i})_{1\leq i \leq d}$ where $\{Z^{(i)}\}_{1\leq i \leq d}$ are independent elephant random walks on $\ZZ$ with parameter 0. Then, since $x_i/n \in [\frac{1}{d}-\varepsilon, \frac{1}{d}+\varepsilon]$ for all $x\in B_n(\varepsilon)$, by Theorem \ref{cltp}, as in (\ref{ERWnuprobbd}), for some positive constant $C(\varepsilon,\nu,d)$ independent of $n$, we have
\begin{equation}
  \label{ERWnnuest}
  \PP(|Z^{(i)}_{x_i}| \leq n^{\nu}) \leq \PP(|\frac{\sqrt{3} Z^{(i)}_{x_i}}{\sqrt{x_i}}| \leq \frac{\sqrt{3}x_i^{\nu}}{(\frac{1}{d}-\varepsilon)^{\nu}\sqrt{x_i}} ) \leq \frac{C(\varepsilon,\nu,d)}{n^{\frac{1}{2}-\nu}}
\end{equation}
Thus, by (\ref{tildeSfirsttermbd}), the first probability on the right-hand side of (\ref{decompPtilSnnu}) is upper bounded by
$$
(\frac{C(\varepsilon,\nu,d)}{n^{\frac{1}{2}-\nu}})^d \PP((b_n(i))_{1\leq i\leq d}\in B_n(\varepsilon))\leq \frac{(C(\varepsilon,\nu,d))^d}{n^{d(\frac{1}{2}-\nu)}}
$$
which is summable if $\nu \in (0,\frac{1}{2}-\frac{1}{d})$. For the second probability, recall $\eta_n(i)$ defined in (\ref{defetan}). By Chebyshev's inequality, for $r\geq 1$,
$$
\PP((b_n(i))_{1\leq i\leq d}\notin B_n(\varepsilon)) \leq \sum_{i=1}^d \PP(|\frac{b_n(i)}{n}-\frac{1}{d}| >\varepsilon ) \leq \frac{\sum_{i=1}^d  \EE |\eta_n(i)|^r }{\varepsilon^r}
$$
By Lemma \ref{ratecon1overd}, for any $p\in [0,1)$, the right-hand side is summable if we choose $r$ large enough, which implies that the right-hand side of (\ref{decompPtilSnnu}) is summable.
\end{proof}

\subsection{The limit in the superdiffusive regime}

Recall the continuous-time embedding $(U_t)_{t\geq 0}$ defined in (\ref{defUt}) and $\xi$, $W_d$ from (\ref{limMERWxiW}).

As mentioned in Remark \ref{fixedpointpaper}, Gu{\'e}rin, Laulin and Raschel proved independently in a recent paper \cite{guerin2023fixed} similar results in this section for $Y_d$, in particular, Corollary \ref{momentWYrec} and Proposition \ref{momcharWY1} (see Theorem 1.4 and Section 2.4 in their paper). Notice that $W_d$ is different from $W$ in Equation (3.7) in \cite{guerin2023fixed} which is a limit in a discrete-time system.

\subsubsection{Proof of Proposition \ref{Lnot0MERW}}
\label{secLneq0}

This section is devoted to the proof of Proposition \ref{Lnot0MERW}.  Recall $a=(2dp-1)/(2d-1)$.

\begin{proposition}
  \label{fgodepropw}
  For $p\in (p_d,1]$, let $\varphi(x)=f(x)+g(x)\mathrm{i}$ be the characteristic function of $w$ defined by $w:=\sum_{i=1}^d W_d(i)$. Then, $\varphi$ satisfies the ODE in (\ref{wcharode}), i.e.
  $$
    \varphi(x)+ a x \varphi^{\prime}(x)=\frac{dp+d-1}{2d-1} \varphi^2(x)+\frac{d-dp}{2d-1}|\varphi(x)|^2, \quad x\in \RR
$$
  In particular, $f$ and $g$ solve the following ODEs
\begin{equation} 
   \label{fgodew}
   \left\{\begin{aligned}  f(x)+ a x f^{\prime}(x)&=f^2(x)+\frac{1-2dp}{2d-1}g^2(x) \\ g(x)+a x g^{\prime}(x)&=\frac{2(dp+d-1)}{2d-1}  f(x) g(x) 
  \end{aligned}  \right.
\end{equation}
with initial conditions $f(0)=1$ and $g^{\prime}(0)=1$.
\end{proposition}
\begin{example}
    \label{p1charvarphi}
  If $p=1$, then $Y_d=e_1$ and $a=1$. By (\ref{limMERWxiW}) and (\ref{WYxi}), $w$ has an Exp(1)-distribution and thus $\varphi(x)=(1-\mathrm{i} x)^{-1}$ which solves (\ref{wcharode}).
\end{example}
\begin{proof}
  We adapt the method developed in the proof of Proposition 4.2, \cite{chauvin2011limit} or Theorem 3.9, \cite{janson2004functional}.  Since $U_0=e_1\in \ZZ^{2d}$, by symmetry (one may interchange $U_t(i)$ and $U_t(-i)$ for $i\neq 1$), 
  $$
  (W_d(1),-W_d(2),-W_d(3),\cdots,-W_d(d)) \stackrel{\mathcal{L}}{=} W_d
  $$
  In particular, $w$ has the same distribution as $W_d(1)-\sum_{i=2}^d W_d(i)$. Observe that if $U$ has initial composition $U_0=(0,1,0,0,\cdots,0) \in \ZZ^{2d}$ which corresponds to a MERW with $S_1=-e_1$, then the corresponding second limit in (\ref{limMERWxiW}), denoted by $W_d^{-}$, satisfies
  $$
  W_d^{-} \stackrel{\mathcal{L}}{=}  (-W_d(1),W_d(2),W_d(3),\cdots,W_d(d))
  $$
  whence the corresponding sum $w^{-}:=\sum_{i=1}^d W_d^{-}(i)$ has the same distribution as $-w$.
  Similarly, for $i \in \{2,3,\cdots,d\}$, given that $U_0=e_{2i-1} \in \ZZ^{2d}$ (or $U_0=e_{2i} \in \ZZ^{2d}$, respectively), by symmetry, the corresponding sum has the same distribution as $w$ (or $-w$, respectively). 

  We define a Bernoulli random variable $\alpha$ such that $\alpha=1$ if a ball of $e_i$ for some $i\in \{1,2,\cdots,d\}$ is added at time $\tau_1$, and $\alpha=0$ if a ball of $-e_i$ for some $i$ is added at time $\tau_1$. Note that $\alpha$ and $\tau_1$ are independent and $\tau_1 \sim \operatorname{Exp}(1)$, $\alpha \sim \operatorname{Bernoulli}(\frac{dp+d-1}{2d-1})$. Starting from time $\tau_1$, we have two independent branching Markov processes as in (\ref{defUt}) with initial compositions $e_1 \in \ZZ^{2d}$ and $e_{2i-1} \in \ZZ^{2d}$ (or $e_{2i} \in \ZZ^{2d}$, respectively) if a ball of $e_i$ (or $-e_i$, respectively) is added at $\tau_1$ for some $i \in \{1,2,\cdots,d\}$. Thus, given that $\alpha=1$, the conditional law of $w$ is 
  $$
  \mathcal{L}(w |\alpha=1) = \mathcal{L}(\lim_{t\to \infty}e^{- a\tau_1}\frac{\sum_{i=1}^d U_{t}(i)-U_{t}(-i)}{e^{a(t-\tau_1)}}|\alpha=1) =  \mathcal{L}(e^{- a\tau_1}([2] w))
  $$
  where the notation $[n] w$ stands for the sum of $n$ independent copies of $w$. Similarly, the conditional law of $w$ given that $\alpha=0$ is equal to the law of 
  $$e^{- a\tau_1}([1] w+[1](-w))$$
  where the notation $[n] X+[m] Y$ stands for the sum of $n$ independent copies of a random variable $X$ and $m$ independent copies of a random variable $Y$. In summary,
  \begin{equation}
    \label{markovequW}
    w \stackrel{\mathcal{L}}{=} e^{-a\tau_1}([1+\alpha] w+[1-\alpha] (-w))
  \end{equation}
  Then, by (\ref{markovequW}), we have, for $x\in \RR$,
  $$
    \begin{aligned}
    \varphi(x) &= \EE e^{\mathrm{i} x w } =\mathbb{E}\left(\mathbb{E}\left(\frac{dp+d-1}{2d-1}e^{\mathrm{i} x e^{-a\tau}([2] w)}+\frac{d(1-p)}{2d-1}e^{\mathrm{i} x e^{-a\tau}([1] w+[1] (-w))} \mid \tau\right)\right) \\
  & =\int_0^{+\infty} \left[\frac{dp+d-1}{2d-1}\varphi^{2}\left(x e^{-at}\right)+\frac{d(1-p)}{2d-1}|\varphi|^2\left(x e^{-a t}\right)\right] e^{-t} dt
  \end{aligned}
  $$
  A change of variable under the integral gives, for $x\neq 0$,
  $$
  \varphi(x)=\frac{x}{a|x|^{1+\frac{1}{a}}} \int_0^x [\frac{dp+d-1}{2d-1}\varphi^{2}(t)+\frac{d(1-p)}{2d-1}|\varphi|^2(t)] \frac{d t}{|t|^{1- \frac{1}{a}}}
  $$
  Differentiation of this equality leads to (\ref{wcharode}) which still holds for $x=0$. Recall that $W_d$ in (\ref{limMERWxiW}) satisfies $\EE W_d=e_1$. The initial conditions are then given by $\varphi(0)=1$ and $\varphi^{\prime}(0)=\mathrm{i}\EE w =\mathrm{i}\EE W_d(1)=\mathrm{i} $.
\end{proof}

Note that $w$ admits a density if $p=1$, see Example \ref{p1charvarphi}. For $p \in (p_d,1)$, using Proposition \ref{fgodepropw}, we prove the following proposition which implies Proposition \ref{Lnot0MERW}.

\begin{proposition}
  \label{wdensity}
For $p\in (p_d,1)$, i.e. $a\in (1/2,1)$, let $\varphi$, $f$ and $g$ be as in Proposition \ref{fgodepropw}. We have 
\begin{equation}
  \label{fgdL1estimate}
  \limsup_{|x| \to \infty} |x^{\frac{1}{a}}f(x)| <\infty, \quad \limsup_{|x| \to \infty} |x^{\frac{1}{a}}g(x)| <\infty 
\end{equation}
In particular, $\sup_{x\in \RR}|x^{\frac{1}{a}}\varphi(x)|<\infty$ and $|\varphi| \in L^1$. Thus, $w$ admits a density $p_w$ given by
$$
p_w(x):=\frac{1}{2 \pi} \int_{\mathbb{R}} \mathrm{e}^{-\mathrm{i} x \cdot z} \varphi(z) d z, \quad x\in \RR
$$
whence $\PP(W_d=0)=\PP(Y_d=0)=0$. Moreover, the distribution of $W_d$ is infinitely divisible.
\end{proposition}
That the distribution of $W_d$ is infinitely divisible can be deduced from Theorem 3.9 \cite{janson2004functional}, see also Remark 3.9 \cite{chauvin2011limit}. Here we provide a direct proof for completeness.
\begin{proof}
  For $x>0$, let $F(x):=x^{\frac{1}{a}}f(x)$ and $G(x):=x^{\frac{1}{a}}g(x)$. Then, for $x>0$, (\ref{fgodew}) reads
  \begin{equation} 
    \label{FGodew}
    \left\{\begin{aligned} F^{\prime}(x)&=\frac{1}{a x^{1+ \frac{1}{a} }}\left(F^2(x)+\frac{1-2dp}{2d-1} G^2(x)\right)  \\ G^{\prime}(x)&=\frac{2(dp+d-1)}{a(2d-1)x^{1+ \frac{1}{a} } }F(x) G(x)
   \end{aligned}  \right.
 \end{equation}
Since $g(0)=0, g^{\prime}(0)=1$, we see that $g(x)>0$ on $ (0,\delta_1)$ for some $\delta_1>0$, and thus $G(x)>0$ on $(0,\delta_1)$. The second equation in (\ref{FGodew}) gives 
\begin{equation}
  \label{soluGbyFw}
  G(x)= G(x_0)\exp(\int_{x_0}^x \frac{2 (dp+d-1) F(t)}{a (2d-1) t^{1+ \frac{1}{a} }}dt)
\end{equation}
for some $x_0\in (0,\delta_1)$ whence $G(x)>0$ for all $x>0$. Since $f(0)=1$, $F>0$ on $ (0,\delta_2)$ for some $\delta_2>0$. We claim that either $F(x)>0$ for all $x>0$ or $F(x)<0$ for all large $x>0$. 

\textbf{Proof of the claim}: If $F(t_0)=0$ for some $t_0>0$, then by (\ref{FGodew}), $F^{\prime}(t_0)<0$ so that we can find $t_1>t_0$ with $F(t_1)<0$. Let $t_2:=\inf\{t>t_1: F(t)=0\}$ with the convention that $\inf \emptyset=\infty$. If $t_2<\infty$, then $F^{\prime}(t_2)\geq 0$ which contradicts (\ref{FGodew}). The claim is proved.

In the first case, i.e. $F>0$ on $(0,\infty)$, since $1-2dp<0$, for any $x>0$,
$F^{\prime}(x)< a^{-1}F^2(x)x^{-1- \frac{1}{a}}$. Integrating it yields the following inequality
$$\frac{1}{F(x)}-\frac{1}{F(x_1)}\geq x^{-\frac{1}{a}} - x_1^{-\frac{1}{a}}, \quad x \geq x_1$$ 
where we choose $x_1>0$ such that $f(x_1)<1$ (this is possible since $f \leq 1$ and $f$ is non-constant). Since $F(x_1) < x_1^{\frac{1}{a}}$, we see that $F(x)$ is upper bounded. In the second case, using (\ref{soluGbyFw}), we see that $G^2$ is upper bounded and thus, for some positive constant $C$, 
$$
F^{\prime}(x)\geq  \frac{F^2(x)-C}{a x^{1+ \frac{1}{a} }}\geq  \frac{-C}{ax^{1+ \frac{1}{a} }}
$$
whence $F$ is lower bounded. Therefore, in either case, $|F|$ is bounded for $x>0$ and so is $|G|$ by (\ref{soluGbyFw}). 

Similarly, we define $\tilde{F}(x):=x^{\frac{1}{a}}f(-x)$ and $\tilde{G}(x):=x^{\frac{1}{a}}g(-x)$ for $x>0$. Then, again, $\tilde{F}$ and $\tilde{G}$ satisfy (\ref{FGodew}). Similarly as in (\ref{soluGbyFw}), we can show that $\tilde{G}(x)<0$ for all $x>0$. And using similar arguments, we can prove that $|\tilde{F}|$ and $|\tilde{G}|$ are bounded on $(0,\infty)$, which completes the proof of (\ref{fgdL1estimate}). 

Now we show that the distribution of $W_d$ is infinitely divisible. Note that we may allow $(U_t, t\geq 0)$ defined in  (\ref{defUt}) to start from any initial composition $(x,,0,\cdots, 0)$, where $x> 0$ is not necessarily integer-valued: Initially, this urn has only one ball of direction $e_1$. At any moment, this special ball is equipped with weight $x$ and an Exp(x)-distributed clock, which means each time its clock rings, the clock will be replaced by a new Exp(x)-distributed clock independent of all the other clocks. When the clock of a ball of a direction, say $e_i$, rings, we add a ball of $e_i$ to the urn with probability $p$ or add a ball of the $2 d-1$ remaining directions each with probability $(1-p) /(2 d-1)$. 
Each ball added after time 0 is equipped with weight 1 and an Exp(1)-distributed clock, all the clocks being independent. Let $U^{(x)}_{t}(\pm i)$ be the total weights of $\pm e_i$ in the urn at time $t$, respectively. Define $U^{(x)}_t:=(U^{(x)}_{t}(1),U_{t}^{(x)}(-1),\cdots,U^{(x)}_{t}(d),U^{(x)}_{t}(-d))$. Note that $(U^{(1)}_t,t\geq 0)$ is simply the process $(U_t,t\geq 0)$ defined in  (\ref{defUt}). Now, observe that for any $n\geq 1$,
$$(U^{(1)}_t)_{t\geq 0} \stackrel{\mathcal{L}}{=}  ([n]U^{(\frac{1}{n})}_t)_{t\geq 0} $$
In particular, the distribution of $W_d$ is infinitely divisible. 
\end{proof}

\begin{proof}[Proof of Corollary \ref{angularasym}]
  (i) If $p>p_d$, by (\ref{superlimY}) and Proposition \ref{Lnot0MERW}, $\lim_{n\to \infty}\hat{S_n}=\hat{Y_d}$ a.s.\\
(ii) From the proof of Theorem 3.2 \cite{bercu2019multi}, especially the equation below Equation (5.17) with $u=e_i$, we see that if $p<p_d$, then for any $i\in \{1,2,\cdots,d\}$, we have the law of the iterated logarithm for $(S_n(i))_{n\in \NN}$: 
$$
\limsup _{n \rightarrow \infty}\frac{S_n(i)}{\sqrt{2 n \log \log n}}=-\liminf _{n \rightarrow \infty}\frac{S_n(i)}{\sqrt{2 n \log \log n}}=\frac{1}{\sqrt{d(1-2a)}}, \quad a.s.
$$
In particular, almost surely, $S_n(i)=0$ infinitely often. This shows that almost surely, the limit of $\hat{S}_n$ does not exist. If $p=p_d$ $(d\geq 2)$, we also have the law of iterated logarithm for $(S_n(i))_{n\in \NN}$, see the proof of Theorem 3.5 \cite{bercu2019multi}. The rest of the proof is similar.
\end{proof}

\subsubsection{Distribution of $Y_1$}
\label{dismomenY1}
In this section, we study the distribution of $Y_1$ in (\ref{superlimY}). For $d=1$, Proposition \ref{fgodepropw} implies the following result.
\begin{corollary}
   \label{fgodeprop}
   For $p\in (3/4,1]$, let $\varphi(x)=f(x)+g(x)\mathrm{i}$ be the characteristic function of $W_1$ defined in (\ref{limMERWxiW}) with $d=1$. Then, $\varphi$ satisfies the following ODE 
  \begin{equation}
    \label{Wcharode1d}
    \varphi(x)+ (2p-1) x \varphi^{\prime}(x)=p \varphi^2(x)+(1-p)|\varphi(x)|^2, \quad x\in \RR
  \end{equation}
In particular, $f$ and $g$ solve the following ODEs
  \begin{equation}
  \label{fgode}
\left\{\begin{array}{l}{f(x)+(2 p-1) x f^{\prime}(x)=f^2(x)+(1-2 p) g^2(x)} \\ {g(x)+(2 p-1) x g^{\prime}(x)=2 p f(x) g(x)}\end{array}\right.
\end{equation}
with initial conditions $f(0)=1$ and $g^{\prime}(0)=1$.
\end{corollary} 
\begin{remark}
  \label{p34proincorrect}
  Coletti and Papageorgiou claimed in Theorem 3.3 \cite{coletti2021asymptotic} that $Y_1$
admits a density by results from \cite{chauvin2011limit}. However, the urn scheme studied in \cite{chauvin2011limit} is different from ours. In their case, given that a “+1” ball is drawn, the respective number of “+1” balls and “-1” balls to be added is deterministic. However, in our case, a “+1” ball is added with probability $p$ and a “-1” ball is added otherwise. As a result, the differential equation (\ref{Wcharode1d}) is different from that in Proposition 5.1 \cite{chauvin2011limit}. They solved that equation with tools from complex analysis which are not applicable in our case due to the non-holomorphic term $|\varphi|^2$. Hence, results in \cite{chauvin2011limit} do not apply to the ERW and a new proof is needed. Unlike Theorem 6.7 \cite{chauvin2011limit}, Proposition \ref{wdensity} provides some estimates on $f$ and $g$, which enables us to conclude that $\varphi \in L^1$ without finding the explicit solution.

Lastly, Theorem 7.4 \cite{chauvin2011limit} proved that the limit of a continuous-time process which corresponds to our $W_1$ admits a density, rather than the discrete-time limit $Y_1$. 
\end{remark}

We may derive the moments of $Y_1$ and $W_1$ from (\ref{fgode}). First, by Theorem 3.5 \cite{pouyanne2008algebraic} or Theorem 3.9 \cite{janson2004functional}, one can show that $Y_1$ is in $L^r$ for any $r\geq 1$, and thus, by (\ref{WYxi}), $W_1 \in L^r$ for any $r\geq 1$. Moreover,
\begin{equation}
  \label{momentWY}
  \EE W_1^r = \EE Y_1^r \EE \xi^{(2p-1)r} =\Gamma((2p-1)r+1)\EE Y_1^r =(2p-1)r\Gamma((2p-1)r)\EE Y_1^r 
\end{equation}
Note that one can retrieve (\ref{momentWY}) by applying Theorem 3.26 \cite{janson2004functional} to the ERW model. This implies that $f$ and $g$ are infinitely differentiable and $f^{(n+1)}(0)=0,g^{(n)}(0)=0$ for even $n\geq 0$.

\begin{corollary}
  \label{momentWYrec}
  For dimension $d=1$ and parameter $p \in (3/4,1)$, denote by $r_n=\EE W_1^n$, $n\in \NN$. Then, $r_1=1$ and 
  \begin{enumerate}[topsep=0pt, partopsep=0pt, leftmargin=5pt, align=left,  label=(\roman*)]
\item For any odd $n \geq 1$, 
  $$
[(n+1)(2p-1)-1]r_{n+1}=2 \sum_{i=1}^{\frac{n-1}{2}} C_n^{2i-1} r_{2i}r_{n+1-2i} +2(2p-1)\sum_{i=1}^{\frac{n+1}{2}} C_n^{2i-1}r_{2i-1}r_{n+2-2i}
  $$
  \item For any even $n \geq 1$, 
$$
n(2p-1)r_{n+1}=2p \sum_{i=1}^{\frac{n}{2}} (C_n^{2i-1}+C_n^{2i}) r_{2i}r_{n+1-2i}
$$
  \end{enumerate}
\end{corollary}
\begin{remark}
  The first four moments of $Y_1$ were given in Theorem 3.8, \cite{bercu2017martingale}. One can deduce from Corollary \ref{momentWYrec} and (\ref{momentWY}) that $\EE Y_1^5=\frac{60p(16p^2-9p-1)}{(4p-3)^2(8p-5)\Gamma(10p-4)}, \cdots$. As mentioned in that paper, the first four moments of $Y_1$ and $W_1$ imply that they are non-Gaussian.
\end{remark}
\begin{proof}[Proof of Corollary \ref{momentWYrec}]
  Let $f$, $g$ be as in Corollary \ref{fgodeprop} with $f(0)=1$ and $g^{\prime}(0)=1$. For $n\geq 1$, by differentiating each side of (\ref{fgode}) $n$ times and applying the Leibniz rule, one can show that for any odd $n \geq 1$, 
    $$
    \begin{aligned}
     [(n+1)(2p-1)-1]f^{(n+1)}(0)&=2 \sum_{i=1}^{\frac{n-1}{2}} C_n^{2i-1} f^{(2i)}(0)f^{(n+1-2i)}(0)\\
     &+2(1-2p)\sum_{i=1}^{\frac{n+1}{2}} C_n^{2i-1}g^{(2i-1)}(0)g^{(n+2-2i)}(0) 
    \end{aligned}
    $$
and for any even $n \geq 1$, 
  $$
  n(2p-1)g^{(n+1)}(0)=2p \sum_{i=1}^{\frac{n}{2}} (C_n^{2i-1}+C_n^{2i}) f^{(2i)}(0)g^{(n+1-2i)}(0)
  $$
By dominated convergence theorem, $\mathrm{i}^{n-1}r_n=g^{(n)}(0)$ for odd $n\geq 1$ and $\mathrm{i}^n r_n=f^{(n)}(0)$ for even $n\geq 1$, which implies the desired result. 
\end{proof}

The following result says that these moments indeed characterize the distributions of $W_1$ and $Y_1$. 
\begin{proposition}
 \label{momcharWY1}
 Let $(r_n)_{n\in \NN}$ be as in Corollary \ref{momentWYrec}. Then, 
  \begin{equation}
    \label{rnbd}
    |r_n| \leq (\frac{p}{2p-1})^{n-1} n!, \quad n\geq 1
  \end{equation}
  In particular, the distributions of $W_1$ and $Y_1$ are determined by the moments given in Corollary \ref{momentWYrec}. Moreover, the characteristic functions $\varphi$ and $\phi$ of $W_1$ and $Y_1$ are given by 
  \begin{equation}
    \label{chard1powerseries}
    \varphi(x) =\sum_{n=0}^{\infty} \frac{r_n\mathrm{i}^n }{n !}x^n,\ |x| < \frac{2p-1}{p}; \quad \phi(x)=\sum_{n=0}^{\infty} \frac{r_n\mathrm{i}^n }{(2p-1)n \Gamma((2p-1)n ) n !}x^n , \ x\in \RR
  \end{equation}
  In addition, the distribution of $W_1$ is infinitely divisible and its density function $p_W$ is supported on the whole real line.
\end{proposition}
\begin{proof}
  We argue by induction. Note that (\ref{rnbd}) is true when $n=1$. Now assume that (\ref{rnbd}) holds for all $n\leq k$ for some $k\geq 1$. If $k$ is odd, then by Corollary \ref{momentWYrec} 
  $$
  \begin{aligned}
      |r_{k+1}| &\leq \frac{2 k! }{(2p-1)k+2p-2}(\frac{p}{2p-1})^{k-1}[\sum_{i=1}^{\frac{k-1}{2}} (2i) +(2p-1)\sum_{i=1}^{\frac{k+1}{2}} (k+2-2i)] \\
      &= (k+1)!(\frac{p}{2p-1})^{k-1} \frac{pk +p-1 }{(2p-1)k+2p-2}\leq (k+1)!(\frac{p}{2p-1})^{k}
  \end{aligned}
  $$
  If $k$ is even, similarly we have 
  $$
  |r_{k+1}| \leq  \frac{2p k! }{(2p-1)k}(\frac{p}{2p-1})^{k-1}[\sum_{i=1}^{\frac{k}{2}} (2i) +\sum_{i=1}^{\frac{k}{2}} (n+1-2i)]= (k+1)!(\frac{p}{2p-1})^{k}
  $$
  Therefore, (\ref{rnbd}) holds for all $n \geq 1$. In particular, the power series $\sum_k r_k x^k / (k!)$ has a positive radius of convergence lower bounded by $(2p-1)/p$. By Theorem 30.1 \cite{billingsley2012probability}, the distribution of $W_1$ is determined by its moments $(r_n)_{n\in \NN}$ and for any $|x| < (2p-1)/p$, $\varphi(x)$ is given by (\ref{chard1powerseries}). By (\ref{momentWY}), the results for $Y_1$ can be proved similarly and note that by Stirling's approximation, the radius of the second power series in (\ref{chard1powerseries}) is infinity.

  By Proposition \ref{wdensity}, the distribution of $W_1$ is infinitely divisible and $W_1$ admits a density $p_W$. We adapt the proof of Proposition 7.1 \cite{chauvin2011limit} to show that $p$ is supported on the whole real line. General results on infinite divisibility (see e.g. Theorem 8.4 \cite{steutel2004infinite}) ensure that the support of an infinitely divisible random variable having a continuous probability distribution function is either a half-line or $\RR$. We argue by contradiction. Suppose that the support of $W_1$ is $[\alpha,\infty)$ for some $\alpha \in \RR$. By (\ref{markovequW}), conditional on $U_{\tau_1}=(1,1)$, $W_1$ is equal in law to 
  $$e^{- (2p-1)\tau_1}([1] W_1+[1](-W_1))$$
  Let $F_W(x)$ be the distribution function of $W_1$. Note that $U_{\tau_1}=(1,1)$ occurs with probability $(1-p)$ and the density function of $-W_1$ is $p_W(-x)$. Then, for any $z<\alpha$, we have
  $$
  \begin{aligned}
    F_W(z) &\geq (1-p) \PP( e^{- (2p-1)\tau_1}([1] W_1+[1](-W_1)) \leq z) \\
       &=(1-p)  \PP( [1] W_1+[1](-W_1) \leq e^{(2p-1)\tau_1}z) \\
       &= (1-p) \int_0^{\infty} e^{-t}  \int_{-\infty}^{\infty} F_W( e^{(2p-1)t}z-y) p_W(-y) d y   dt \\
       &=(1-p) \int_0^{\infty} e^{-t}  \int_{\alpha}^{\infty} F_W( e^{(2p-1)t}z+x) p_W(x) d x   dt>0
  \end{aligned}
  $$
  which contradicts our assumption. Similarly, the support of $p$ can not be of the form $(-\infty,\beta]$ for some $\beta \in \RR$. This completes the proof.
\end{proof}

\subsection{MERW on $\ZZ^2$: Lyapunov functions method}
\label{2dERWsec}

\subsubsection{Proof of Proposition \ref{Snsio}}
\label{proofweakerrec}

\begin{proof}[Proof of Proposition \ref{Snsio}]
  For $k\geq 1$, define 
  $$\tau_k:=\inf\{j\geq k: \|S_j\| \leq j^s \}$$
  We claim that for any $k\geq 1$, $\tau_k<\infty$ a.s.. Note that the claim implies the desired result.

  \textbf{Proof of the claim}: We may assume that $\|S_k\| > k^s$ and $k$ is large enough. For $x\in \ZZ^2$ with $\|x\|\geq 1$, we let $f(x):=\sqrt{\log \|x\|}$. Note that the function $f$ was used in Section 2.3 \cite{popov2021two} to prove the recurrence of the 2-dimensional SRW. For $e =\pm e_i$, $i=1,2$,
  $$
\begin{aligned}
  \sqrt{\log \|x+e\|}-\sqrt{\log \|x\|}&=\sqrt{\log \|x\|} \left(\left(1+\frac{1}{\log \|x\|^2}\log \left(1+\frac{2 x \cdot e+1}{\|x\|^2}\right)\right)^{1/2}-1\right)
\end{aligned}
  $$ 
  Using Taylor expansions: as $x\to 0$,
  \begin{equation}
    \label{taylorexp}
      \log(1+x)= x-\frac{x^2}{2}+O(|x|^3); \quad \sqrt{1+x}=1+\frac{x}{2}-\frac{x^2}{8}+O(|x|^3)
  \end{equation}
   one can show that for $x \in \ZZ^2$ and $e = \pm e_i$, $i=1,2$, as $\|x\| \to \infty$,
  \begin{equation}
    \label{taylorsqrtln1x}
     \begin{aligned}
    \sqrt{\log \|x+e\|}-\sqrt{\log \|x\|}&=\sqrt{\log \|x\|}\left(\frac{1}{2 \log \|x\|}\left(\frac{ x \cdot e}{\|x\|^2}+\frac{1}{2\|x\|^2}-\frac{(x \cdot e)^2}{\|x\|^4}+O\left(\|x\|^{-3}\right)\right)\right. \\ &\left.-\frac{1}{8 \log^2 \|x\|} \frac{(x \cdot e)^2}{\|x\|^4}+O\left(\|x\|^{-3}(\log \|x\|)^{-2}\right)\right) 
  \end{aligned}
  \end{equation}
   Then, by (\ref{taylorsqrtln1x}) with $x=S_n,e=\sigma_{n+1}$ and Lemma \ref{computationlem} (ii), (iii), we have
  $$
  \begin{aligned}
    &\quad\ \mathbb{E}\left[f\left(S_{n+1}\right)-f\left(S_n\right) \mid \FF_n\right] \\
    &=\frac{a}{2n\sqrt{\log\|S_n\|}}-\frac{1}{2\|S_n\|^2 \log^{3/2}\|S_n\|}\left(\frac{a \log \|S_n\| }{\|S_n\|^2}\sum_{i=1}^2 (\frac{b_n(i)}{n}-\frac{1}{2})S_n(i)^2 \right. \\ 
    &\left.\quad\ +\frac{1}{8}+\frac{a}{4\|S_n\|^2}\sum_{i=1}^2 (\frac{b_n(i)}{n}-\frac{1}{2})S_n(i)^2+O\left(\frac{\log \|S_n\|}{\|S_n\|}\right)\right)
  \end{aligned}
$$
Since $a<1/2$ and $S_n(i)^2 \leq \|S_n\|^2$, we have 
\begin{equation}
  \label{2dEfdiff}
  \begin{aligned}
    &\quad\ \mathbb{E}\left[f\left(S_{n+1}\right)-f\left(S_n\right) \mid \FF_n\right]-\frac{a}{2n\sqrt{ \log \|S_n\|}} \\
    &\leq -\frac{1}{2\|S_n\|^2 \log^{3/2}\|S_n\|}\left(\frac{1}{10}-\log \|S_n\|\sum_{i=1}^2|\frac{b_n(i)}{n}-\frac{1}{2}|\right)
  \end{aligned}
\end{equation}
  if $S_n \notin B(0,r)$ for some large $r$. Note that here we can replace $1/10$ by any number strictly less than $1/8$ if we choose $r$ large enough. 
  
  We first assume that $a\geq 0$. If $k+n+1\leq \tau_k$,
  \begin{equation}
    \label{upbdsumajS}
      \sum_{j=k}^{k+n}\frac{a}{2j\sqrt{\log \|S_j\|}} \leq \sum_{j=k}^{k+n } \frac{a}{2\sqrt{s}j \sqrt{\log j}} \leq  \frac{a}{2 \sqrt{s}} \int_{k-1}^{k+n} \frac{1}{x \sqrt{\log x}} dx  < \sqrt{s \log (k+n)} < \sqrt{\log \|S_{k+n}\|}
  \end{equation}
 Define 
  $$
  T_k:=\inf\{j \geq k:  \ \sum_{i=1}^2|\frac{b_j(i)}{j}-\frac{1}{2}|\geq \frac{1}{j^{1/4}}\}
  $$
Note that if $k+n +1\leq  T_k$, then 
\begin{equation}
  \label{lnSdiffsmalll}
  \log \|S_{k+n}\|\sum_{i=1}^2|\frac{b_{k+n}(i)}{k+n}-\frac{1}{2}| \leq \frac{\log \|S_{k+n}\|}{(k+n)^{\frac{1}{4}}}\leq \frac{\log (k+n)}{(k+n)^{\frac{1}{4}}} <\frac{1}{10}
\end{equation}
if we assume that $k$ is large enough. Thus, by (\ref{2dEfdiff}), (\ref{upbdsumajS}) and (\ref{lnSdiffsmalll}), we see that $\{Y_n\}_{n\in \NN}$ defined by $Y_0=f(S_k)$ and
  $$
 Y_{n+1}:=f(S_{(k+n+1)\wedge \tau_k \wedge T_k}) -   \sum_{j=k}^{(k+n+1)\wedge \tau_k \wedge T_k-1}\frac{a}{2j\sqrt{\log \|S_j\|}} ,\quad  n\in \NN
  $$
  is a supermartingale. Since $f\left(S_{n+1}\right)-f\left(S_n\right)$ is lower bounded, by (\ref{upbdsumajS}), $\{Y_n\}_{n\in \NN}$ is lower bounded. In particular, $(Y_n)$ converges a.s.. By (\ref{upbdsumajS}) and the law of iterated logarithm for MERWs (\ref{ltilSdiff}), a.s. on $\{\tau_k=\infty\} \cap \{T_k=\infty\}$ 
$$
\begin{aligned}
  \limsup_{n\to \infty} Y_n &\geq \limsup_{n\to \infty} (\sqrt{\log \|S_{k+n+1}\|} - \sqrt{s \log (k+n)} )\\ 
  &=\limsup_{n\to \infty} \sqrt{\log (k+n)} (\sqrt{\frac{\log \|S_{k+n+1}\|}{\log (k+n)}} - \sqrt{s} )=\infty
\end{aligned}
$$
where we used the assumption $s<1/2$. Thus,
  $$
  \PP(\{\tau_k=\infty\} \cap \{T_k=\infty\}) =0
  $$
Indeed, we have proved that for any $m\geq k$,
  \begin{equation}
    \label{taukTmprob0}
      \PP(\{\tau_k=\infty\} \cap \{T_m=\infty\}) \leq \PP(\{\tau_m=\infty\} \cap \{T_m=\infty\}) =0
  \end{equation}
  where we used the fact that $\{\tau_k =\infty\} \subset \{\tau_m=\infty\}$. 
The case $a\leq 0$ is even simpler since by (\ref{2dEfdiff}) and (\ref{lnSdiffsmalll}), $\{f(S_{(k+n+1)\wedge \tau_k \wedge T_k}) \}_{n\in \NN}$ itself is a lower bounded supermartingale, and thus (\ref{taukTmprob0}) is still true. 
  By Lemma \ref{ratecon1overd}, $n^{\frac{1}{4}}(b_n(1)/n - 1/2)$ converges to 0 a.s. whence
  $$
  \PP(\bigcup_{m\geq k} \{T_m=\infty\} )=1
  $$
  Combined with (\ref{taukTmprob0}), this implies that $\PP(\tau_k<\infty)=1$. 
\end{proof}

\subsubsection{Proof of Theorem \ref{phasetranZ2MERW}}
\label{proof2d}
We first prove Proposition \ref{z258logn}, i.e. the critical case.

\begin{proof}[Proof of Proposition \ref{z258logn}]
 For any $n>1$, let $x_n:=\log (\|S_n\|^2+n^{\frac{4}{5}})/\log n$. Then, it is equivalent to proving that
  $$
  \lim_{n\to \infty}x_n=1, \quad a.s.
  $$
Now, using Taylor expansion, we see that there exists an $\varepsilon>0$ such that if $|x|\leq \varepsilon$, then
\begin{equation}
 \label{tayineln1x}
 \log (1+x) \geq x -\frac{1}{2}x^2-|x|^3
\end{equation}
Thus, there exists a positive integer $m>1$ such that for all $n\geq m$,
\begin{equation}
  \label{xnrecur58}
\begin{aligned}
  x_{n+1}-x_n&=\frac{\log \left(1+\frac{2S_n\cdot \sigma_{n+1}+1+(n+1)^{\frac{4}{5}}-n^{\frac{4}{5}}}{\|S_n\|^2+n^{\frac{4}{5}} }\right) }{\log (n+1)}+\log (\|S_n\|^2+n^{\frac{4}{5}}) (\frac{1}{\log (n+1)}-\frac{1}{\log n}) \\
  &\geq \frac{1}{\log (n+1)}\left(u_{n+1}-\frac{1}{2}u_{n+1}^2-|u_{n+1}|^3\right)-\frac{\log (\|S_n\|^2+n^{\frac{4}{5}})}{n \log^2 n}
\end{aligned}
\end{equation}
where 
$$
u_{n+1}:=\frac{2S_n\cdot \sigma_{n+1}+1+(n+1)^{\frac{4}{5}}-n^{\frac{4}{5}}}{\|S_n\|^2+n^{\frac{4}{5}} }
$$
Indeed, for some universal constant $c_1$ independent of $n$,
\begin{equation}
  \label{un25bd}
  |u_{n+1}|\leq \frac{2\|S_n\|}{\|S_n\|^2+n^{\frac{4}{5}} }+ \frac{2}{\|S_n\|^2+n^{\frac{4}{5}} } \leq \frac{c_1}{n^{\frac{2}{5}}}, \quad \forall n>1
\end{equation}
In particular, $|u_{n+1}|<\varepsilon$ for large $n$ (say $n\geq m$) so that we can apply (\ref{tayineln1x}) in (\ref{xnrecur58}).

Since $\{x_n\}$ is upper-bounded, it remains to show that
\begin{equation}
  \label{sumnegabd}
  \EE\left(\sum_{n=m}^{\infty} \EE (x_{n+1}-x_n|\FF_n)^{-} \right)<\infty
\end{equation}
where we recall that $x^{-}=\max (-x, 0)$. Indeed, (\ref{sumnegabd})  implies that 
$$
x_{n+1}+\sum_{j=m}^{n}\EE (x_{j+1}-x_{j}|\FF_j)^{-}, \quad n\geq m
$$
is an $L^1$-bounded submartingale and thus converges a.s.. Moreover, (\ref{sumnegabd}) implies that $\sum_{n=m}^{\infty} \EE (x_{n+1}-x_n|\FF_n)^{-}$ converges a.s. and therefore $x_n$ converges a.s.. By the law of the iterated logarithm for the MERW (\ref{ltilScrit}), we have $\limsup_{n\to \infty}x_n = 1$ a.s. and thus $\lim_{n\to \infty}x_n = 1$.

Since $\|S_n\|^2\leq n^2$, we have
\begin{equation}
  \label{Snlognn1}
  \sum_{n=m}^{\infty}\frac{\log (\|S_n\|^2+n^{\frac{4}{5}})}{n \log n}\left(\frac{1}{\log n}-\frac{1}{\log(n+1)}\right) \leq \sum_{n=m}^{\infty}\frac{\log (n^2+n^{\frac{4}{5}})}{n \log n}\frac{1}{n\log^2 n} <c_2
\end{equation}
for some universal constant $c_2$. By (\ref{xnrecur58}), (\ref{un25bd}) and (\ref{Snlognn1}), to prove (\ref{sumnegabd}), it suffices to show that 
\begin{equation}
  \label{lognunlogS}
  \EE\left(\sum_{n=m}^{\infty}\frac{1}{\log (n+1)} \EE (u_{n+1}-\frac{1}{2}u_{n+1}^2-\frac{\log (\|S_n\|^2+n^{\frac{4}{5}})}{n \log n}|\FF_n)^{-} \right) <\infty
\end{equation}

Using Lemma \ref{computationlem} (iii) with $a=1/2$, we have
$$
|\mathbb{E}\left(\left[S_n \cdot \sigma_{n+1}\right]^2 \mid \mathcal{F}_n\right)-\frac{1}{2}\left\|S_n\right\|^2|\leq \frac{\|S_n\|^2}{2}\sum_{i=1}^2 |\frac{b_n(i)}{n}-\frac{1}{2}| 
$$
Using this inequality and Lemma \ref{computationlem} (ii), we have 
\begin{equation}
  \label{condEunun2}
  \begin{aligned}
  &\quad \EE (u_{n+1}-\frac{1}{2}u_{n+1}^2|\FF_n)\geq \frac{\|S_n\|^2}{n(\|S_n\|^2+n^{\frac{4}{5}}) }+\frac{1+(n+1)^{\frac{4}{5}}-n^{\frac{4}{5}}}{\|S_n\|^2+n^{\frac{4}{5}}} -\frac{\|S_n\|^2}{(\|S_n\|^2+n^{\frac{4}{5}})^2}\\ 
  &-\frac{\sum_{i=1}^2|\frac{b_n(i)}{n}-\frac{1}{2}|}{\|S_n\|^2+n^{\frac{4}{5}}} -\frac{\|S_n\|^2(1+(n+1)^{\frac{4}{5}}-n^{\frac{4}{5}} )}{n(\|S_n\|^2+n^{\frac{4}{5}})^2} -\frac{(1+(n+1)^{\frac{4}{5}}-n^{\frac{4}{5}} )^2}{2(\|S_n\|^2+n^{\frac{4}{5}})^2} \\
  &\geq \frac{\|S_n\|^2}{n(\|S_n\|^2+n^{\frac{4}{5}}) }+\frac{n^{\frac{4}{5}}}{(\|S_n\|^2+n^{\frac{4}{5}})^2}-\frac{\sum_{i=1}^2|\frac{b_n(i)}{n}-\frac{1}{2}|}{\|S_n\|^2+n^{\frac{4}{5}}}-\frac{2\|S_n\|^2}{n(\|S_n\|^2+n^{\frac{4}{5}})^2} -\frac{2}{(\|S_n\|^2+n^{\frac{4}{5}})^2} 
\end{aligned}
\end{equation}
where we used that $0<(n+1)^{\frac{4}{5}}-n^{\frac{4}{5}}\leq 1$. By Lemma \ref{ratecon1overd}, there exists a constant $c_3$ such that for any $n>1$,
$$
  \EE \sum_{i=1}^2|\frac{b_n(i)}{n}-\frac{1}{2}| \leq \frac{c_3 \sqrt{\log n}}{\sqrt{n}}
$$
which implies that
\begin{equation}
  \label{unun23finite}
\EE \left( \sum_{n=m}^{\infty}\frac{\sum_{i=1}^2|\frac{b_n(i)}{n}-\frac{1}{2}|}{\|S_n\|^2+n^{\frac{4}{5}}} \right)<\infty
\end{equation}
Moreover,
\begin{equation}
  \label{unun245finite}
  \sum_{n=m}^{\infty}\frac{2\|S_n\|^2}{n(\|S_n\|^2+n^{\frac{4}{5}})^2}\leq \sum_{n=m}^{\infty}\frac{2}{n^{\frac{9}{5}}} <\infty , \quad  \sum_{n=m}^{\infty}\frac{2}{(\|S_n\|^2+n^{\frac{4}{5}})^2}\leq \sum_{n=m}^{\infty}\frac{2}{n^{\frac{8}{5}}} <\infty 
\end{equation}
By (\ref{condEunun2}), (\ref{unun23finite}) and (\ref{unun245finite}), to prove (\ref{lognunlogS}), it suffices to show that
\begin{equation}
  \label{sumvn-bd}
  \EE \sum_{n=m}^{\infty} v_n^{-} <\infty
\end{equation}
where
$$
v_n:=\frac{1}{\log (n+1)} \left(\frac{\|S_n\|^2}{n(\|S_n\|^2+n^{\frac{4}{5}}) }+\frac{n^{\frac{4}{5}}}{(\|S_n\|^2+n^{\frac{4}{5}})^2}-\frac{\log (\|S_n\|^2+n^{\frac{4}{5}})}{n \log n} \right)
$$
\textbf{Proof of (\ref{sumvn-bd})}. We divide $\EE \sum_{n=m}^{\infty} v_n^{-}$ into three parts. If $ n^{\frac{9}{10}} \leq \|S_n\|^2 \leq n \log^2 n$, then $\|S_n\|^2+ n^{\frac{4}{5}}\leq n( \log^2 n+1)$ and thus
$$
\begin{aligned}
  \frac{\|S_n\|^2}{\|S_n\|^2+n^{\frac{4}{5}} }-\frac{\log (\|S_n\|^2+n^{\frac{4}{5}})}{ \log n}&\geq 1-\frac{n^{\frac{4}{5}}}{\|S_n\|^2+n^{\frac{4}{5}} } - \frac{\log (n( \log^2 n+1))}{ \log n}\\
  &\geq -\frac{1}{n^{\frac{1}{10}}}-\frac{\log ( \log^2 n+1) }{\log n}
\end{aligned}
$$
whence we have
\begin{equation}
  \label{vn-sn910}
\EE \sum_{n=m}^{\infty}v_n^{-} \mathds{1}_{\{n^{\frac{9}{10}} \leq \|S_n\|^2 \leq n \log^2 n\}} \leq \sum_{n=m}^{\infty} \frac{1}{n\log (n+1)} \left(\frac{1}{n^{\frac{1}{10}}}+\frac{\log ( \log^2 n+1) }{\log n}\right)<\infty
\end{equation}
By Chebyshev's inequality and Corollary \ref{expeSn2} (which implies that $\EE \|S_n\|^2 \sim n \log n$, and will be proved later), for some constant $c_4$, we have
$$
\PP(\|S_n\|^2 > n \log^2 n) \leq  \frac{\EE \|S_n\|^2}{n \log^2 n} \leq \frac{c_4}{\log n}, \quad \forall n\geq m
$$
Observe that $\|S_n\|^2+n^{\frac{4}{5}}\leq 2n^2 \leq n^3$ for any $n>1$. Then,
\begin{equation}
  \label{vn-snlog2n}
\EE \sum_{n=m}^{\infty}v_n^{-} \mathds{1}_{\{\|S_n\|^2 > n \log^2 n\}} \leq \sum_{n=m}^{\infty}\frac{3\PP(\|S_n\|^2 > n \log^2 n)}{n\log (n+1)}   \leq \sum_{n=m}^{\infty} \frac{3c_4}{n \log^2 n} <\infty 
\end{equation}
For all large $n$ with $\|S_n\|^2 < n^{\frac{9}{10}}$,
$$
  \frac{n^{\frac{4}{5}}}{(\|S_n\|^2+n^{\frac{4}{5}})^2}-\frac{\log (\|S_n\|^2+n^{\frac{4}{5}})}{n \log n} >   \frac{n^{\frac{4}{5}}}{(n^{\frac{9}{10}}+n^{\frac{4}{5}})^2} -\frac{\log (n^{\frac{9}{10}}+n^{\frac{4}{5}})}{n\log n}  > 0
$$
which implies that there exists a positive constant $K$ such that for any $n\geq K$,
\begin{equation}
  \label{vnSnleqn910}
v_n \mathds{1}_{\{\|S_n\|^2 < n^{\frac{9}{10}}\}}\geq 0
\end{equation}
Now (\ref{sumvn-bd}) follows from  (\ref{vn-sn910}), (\ref{vn-snlog2n}) and (\ref{vnSnleqn910}).
\end{proof}

Now we are ready to prove Theorem \ref{phasetranZ2MERW}.
\begin{proof}[Proof of Theorem \ref{phasetranZ2MERW}]
  Proposition \ref{z258logn} and Proposition \ref{Lnot0MERW} imply that $S$ is transient for $p\geq 5/8$. Now we assume that $p<5/8$. Fix a positive constant $s\in (a,1/2)$. Set $n_1=1$. We define inductively a sequence of stopping times, for $k>1$,
$$
n_k=\inf\{n>n_{k-1}: \|S_{n} \|\leq n^{s},\ |\frac{b_{n}(i)}{n}-\frac{1}{2}| \leq \frac{1}{n^{\frac{1}{3}}},\ \forall i=1,2\}
$$
with the convention that $\inf \emptyset=\infty$. By Lemma \ref{ratecon1overd} and Proposition \ref{Snsio}, almost surely, $n_k<\infty$ for all $k\geq 1$. Again, we use the Lyapunov function $f(x)=\sqrt{\log \|x\|}$. By (\ref{2dEfdiff}),  
\begin{equation}
 \label{fsnrecsqrtlog}
 \begin{aligned}
   &\quad\ \mathbb{E}\left[f\left(S_{n+1}\right)-f\left(S_n\right) \mid \FF_n\right] \\ 
   &\leq \frac{-1}{2\|S_n\|^2 \log ^{3/2}\|S_n\|}\left(\frac{1}{10}-\log \|S_n\|\sum_{i=1}^2|\frac{b_n(i)}{n}-\frac{1}{2}|-\frac{a\|S_n\|^2\log \|S_n\|}{n}\right)
 \end{aligned}
\end{equation}
 if $S_n \notin B(0,r)$ for some large $r$. For each $n_k$, we define the stopping times $$\tau_k:=\inf\{j\geq n_k: S_j \in B(0,r) \}, \quad T_k:=\inf\{j \geq n_k:  \ \sum_{i=1}^2|\frac{b_j(i)}{j}-\frac{1}{2}|\geq \frac{1}{j^{1/4}}\}
 $$
 and choose $t\in (s,1/2)$, define $\theta_k:=\inf\{j \geq n_k: \|S_j\| \geq j^t \}$. Note that for $n_k + n< \theta_k$, 
 $$\frac{\|S_{n_k + n}\|^2\log \|S_{n_k + n}\|}{n_k + n} < \frac{\log (n_k + n)}{(n_k + n)^{1-2t}} < \frac{1}{20}$$
 if we assume that $k$ is large enough. Moreover, as in (\ref{lnSdiffsmalll}), we can show that if $n_k+n < \tau_k \wedge T_k \wedge \theta_k$, then 
 $$ \log \|S_{n_k + n}\|\sum_{i=1}^2|\frac{b_{n_k + n}(i)}{n_k + n}-\frac{1}{2}|< \frac{1}{20}$$
 Then, by (\ref{fsnrecsqrtlog}), $\{f(S_{(n_k+n)\wedge \tau_k \wedge T_k \wedge \theta_k})\}_{n\in \NN}$ is a non-negative supermartingale and thus converges a.s.. By the law of the iterated logarithm for MERWs (\ref{ltilSdiff}), $\theta_k<\infty$ a.s.. Thus, 
 $$\lim_{n\to \infty}f(S_{(n_k+n)\wedge \tau_k \wedge T_k  \wedge \theta_k}) = f(S_{\theta_k}) \quad \text{on}\ \{\tau_k=\infty\} \cap \{T_k=\infty\}$$ 
 Then by the optional stopping theorem, see e.g. Theorem 16, Chapter V \cite{dellacherie1982probabilities},
 $$
 \begin{aligned}
    \sqrt{s}\sqrt{\log n_k} &\geq \EE (f(S_{n_k  })|\FF_{n_k}) \geq \EE (f(S_{\theta_k})\mathds{1}_{\{\tau_k=\infty\} \cap \{T_k=\infty\}}|\FF_{n_k}) \\
    &\geq \EE (\sqrt{\log \theta_k^{t}}\mathds{1}_{\{\tau_k=\infty\} \cap \{T_k=\infty\}}|\FF_{n_k}) \geq \sqrt{t} \sqrt{\log n_k} \PP(\{\tau_k=\infty\} \cap \{T_k=\infty\}|\FF_{n_k})
 \end{aligned}
 $$
 Therefore, for all large $k$,
 \begin{equation}
   \label{inequtkcapTk}
     \PP(\{\tau_k=\infty\} \cap \{T_k=\infty\}|\FF_{n_k}) \leq \sqrt{\frac{s}{t}}
 \end{equation}
  Now we show that 
  \begin{equation}
    \label{TkinfFnk1}
    \lim_{k\to\infty} \PP( T_k=\infty|\FF_{n_k}) = 1
  \end{equation}
 Recall that in the proof of Lemma \ref{ratecon1overd}, we wrote $\eta_n:=\frac{b_n(1)}{n}-\frac{1}{2}$ with a slight abuse of notation. By (\ref{etanformula}), we get
 \begin{equation}
  \label{etankformula}
  \eta_n=\beta_n(\frac{\eta_{n_k}}{\beta_{n_k}}+ \sum_{j=n_k}^{n-1}\frac{\gamma_j}{\beta_{j+1}}\epsilon_{j+1}), \quad n\geq n_k
 \end{equation}
Then by (\ref{betanasy}) and the choice of $n_k$, there exists a constant $C_1$ not depending on $k$ and $n$ such that for any $n\geq n_k$, 
 \begin{equation}
   \label{firtermetanksmall}
    n^{\frac{1}{4}}\frac{\beta_n|\eta_{n_k}|}{\beta_{n_k}}\leq C_1 \frac{n^{\frac{1}{4}}}{n_k^{\frac{1}{3}}} (\frac{n_k}{n})^{\frac{4(1-p)}{3}}\leq C_1 \frac{n^{\frac{1}{4}}}{n_k^{\frac{1}{3}}} (\frac{n_k}{n})^{\frac{1}{2}} \leq C_1\frac{n_k^{\frac{1}{2}-\frac{1}{3}}}{n^{\frac{1}{2}-\frac{1}{4}}} \leq \frac{C_1}{n^{\frac{1}{12}}} 
 \end{equation}
 which is less than $1/4$ if $k$ is large enough. As in (\ref{etannnu}), if $p<5/8$, for $r\geq 1$, conditional on $\FF_{n_k}$, by Burkholder’s inequality, 
\begin{equation}
  \label{estressumgambeta}
  \PP(\beta_n|\sum_{j=n_k}^{n-1}\frac{\gamma_j}{\beta_{j+1}}\epsilon_{j+1}|\geq \frac{1}{4n^{\frac{1}{4}}}) \leq \frac{C_2(r,p) }{n^{\frac{r}{4}}}, \quad \forall n\geq n_k
\end{equation}
where the constant $C_2(r,p)$ does not depend on $k$ and $n$. Conditional on $\FF_{n_k}$, by (\ref{estressumgambeta}) with $r>4$, we have
\begin{equation}
  \label{Tkn14infi}
  \PP(\bigcup_{n\geq n_k} \beta_n|\sum_{j=n_k}^{n-1}\frac{\gamma_j}{\beta_{j+1}}\epsilon_{j+1}|\geq \frac{1}{4n^{\frac{1}{4}}}) \leq \sum_{n\geq n_k}\PP(\beta_n|\sum_{j=n_k}^{n-1}\frac{\gamma_j}{\beta_{j+1}}\epsilon_{j+1}|\geq \frac{1}{4n^{\frac{1}{4}}})  \to 0, 
\end{equation}
as $k\to \infty$. By (\ref{etankformula}), (\ref{firtermetanksmall}) and (\ref{Tkn14infi}), 
$$ \lim_{k\to \infty} \PP(n^{\frac{1}{4}}|\frac{b_n(1)}{n}-\frac{1}{2}|<\frac{1}{2}, \ \forall n\geq n_k |\FF_{n_k}) = 1$$  
We can prove a similar result for $\frac{b_n(2)}{n}-\frac{1}{2}$, which completes the proof of (\ref{TkinfFnk1}). By (\ref{inequtkcapTk}) and (\ref{TkinfFnk1}), for any $p<5/8$, we can find positive constants $K$ and $c$ such that for all $k\geq K$, $
\PP(\tau_k<\infty|\FF_{n_k}) > c$. Fix $k\geq K$, then for any $m\geq k$,
$$
\PP(\tau_k<\infty|\FF_{n_m}) \geq \PP(\tau_m<\infty|\FF_{n_m}) > c>0
$$
whence we have $\tau_k<\infty$ a.s. by Levy's 0-1 law. This shows that almost surely for any $k\geq K$, $S$ will return to $B(0,r)$ after time $n_k$. Thus, $S$ is recurrent by Proposition \ref{merw01}. 
\end{proof}

\section{Some related results}
\label{someotherre}
\subsection{Estimates on the expected exit times}
\label{estexpexit}

We provide some estimates of the exit times of MERWs on $\ZZ^d$. Recall $a$ from  (\ref{superlimY}).

\begin{lemma}
  \label{MartS2}
  Let $S$ be a MERW on $\ZZ^d$ with parameter $p$. Let $i_a=2$ if $a=-1/2$, $i_a=3$ if $a=-1$, and $i_a=1$ otherwise. For $n\geq i_a$, define
  $$
\gamma_n = \left \{ \begin{aligned} &\prod_{i=1}^{n-1} (1+\frac{2a}{i})  && \text{if } a\neq -1/2 \ \text{or} -1\\ &\frac{1}{n-1} && \text{if } a=-1/2 \\ &\frac{2}{(n-1)(n-2)} && \text{if } a=-1 \end{aligned} \right.
$$
  with the convention that $\gamma_1=1$ if $i_a=1$. Define
  $$
M_n:=\frac{\|S_n\|^2}{\gamma_n} -\sum_{i=i_a}^n \frac{1}{\gamma_i}, \quad n\geq i_a; \quad N_n := \|S_n\|^2-n-2a \sum_{\ell=1}^{n-1} \frac{\|S_{\ell}\|^2}{\ell}, \quad n\geq 1
  $$
   with the convention that $N_1=0$. Then $(M_n)_{n\geq i_a}$ and $(N_n)_{n\geq 1}$ are martingales with $\EE M_{i_a}=0$. 
\end{lemma}
\begin{proof}
By Lemma \ref{computationlem} (ii), we have
  \begin{equation}
    \label{proMnmartequ}
    \EE (\|S_{n+1}\|^2|\FF_n) =\|S_n\|^2+1+ 2\EE (S_{n}\cdot \sigma_{n+1}|\FF_n) =(1+\frac{2a}{n})\|S_n\|^2+1
  \end{equation}
from which one can easily deduce the assertions. 
\end{proof}

  Using Lemma \ref{MartS2}, one can easily find the rate of growth of $\EE \|S_n\|^2$. 

  \begin{corollary}
    \label{expeSn2}
    Let $S$ be a MERW with parameter $p$. Then,
    \begin{equation}
      \label{asyEsn2}
         \EE \|S_n\|^2 \sim \left \{ \begin{aligned} &\frac{n}{1-2a}  && \text{if } p<p_d \\ & n \log n && \text{if } p=p_d \\ & \frac{n^{2a}}{(2a-1)\Gamma(2a)} && \text{if } p>p_d \end{aligned} \right.
    \end{equation}
  \end{corollary}
  \begin{remark}
    The corresponding results for the one-dimensional case were derived by Schütz and Trimper in \cite{schutz2004elephants}.
  \end{remark}
  \begin{proof}
    If $a\neq -1/2$ or $-1$, $\gamma_n$ in Lemma \ref{MartS2} equals 
\begin{equation}
  \label{gammandefGam}
\gamma_n=\frac{(1+2a)\Gamma(n+2a)}{\Gamma(2+2a) \Gamma(n)} \sim \frac{1+2a}{\Gamma(2+2a)}n^{2a}
\end{equation}
by the properties of the Euler Gamma function. Thus,
    $$
   \lim_{n \to \infty} \frac{\sum_{i=1}^n \frac{1}{\gamma_i}}{n^{1-2a}} =\frac{\Gamma(2+2a)}{(1-2a)(1+2a)}, \ \text{if} \ p<p_d, \quad   \lim_{n \to \infty}  \frac{\sum_{i=1}^n \frac{1}{\gamma_i}}{\log n} = 1, \ \text{if} \ p=p_d
    $$
    If $p>p_d$, by (\ref{gammandefGam}), the properties of the Beta function $B(\cdot,\cdot)$ and the dominated convergence theorem, we have
    $$
        \sum_{i=1}^{\infty} \frac{1}{\gamma_i}=\sum_{i=1}^{\infty}\frac{2a\Gamma(2a) \Gamma(i)}{\Gamma(i+2a)}=2a\sum_{i=1}^{\infty} B(2a,i) =2a\sum_{i=1}^{\infty}\int_0^1 x^{i-1}(1-x)^{2a-1}dx=\frac{2a}{2a-1}
  $$
  Combined with Lemma \ref{MartS2}, these results prove (\ref{asyEsn2}) when $a\neq -1/2$ or $-1$. The cases $a= -1/2$ ($\gamma_n=1/(n-1)$) and $a=-1$ ($\gamma_n=2/[(n-1)(n-2)]$) are proved similarly.
  \end{proof}

 We may estimate the expected exit times of $S$ by Lemma \ref{MartS2}. 
  \begin{proposition}
    \label{estexpectMERWhit}
    Let $S$ be a MERW on $\ZZ^d$ with parameter $p\in [0,1]$. For $m\geq 1$, let $\zeta_m$ be the exit time from $B(0,m)$ as in (\ref{exittimedef}). 
    \begin{enumerate}[topsep=0pt, partopsep=0pt, leftmargin=5pt, align=left, label=(\roman*)]
 \item For any $p\in [0,1]$, we have $\EE \zeta_m \leq 6(m+1)^2$ 
 \item If $p<p_d$, then there exists a positive constant $C(p,d)$ such that for any $m\geq 1$,
    $$
     \EE \zeta_m \geq C(p,d) m^2
    $$
    \item If $p=p_d$, then there exists a positive constant $C(d)$ such that for any $m\geq 1$,
  $$
  \EE \zeta_m\cdot \log \EE\zeta_m  \geq C(d) m^2
  $$
  \item If $p>p_d$, then there exists a positive constant $C(p,d)$ such that for any $m\geq 1$,
  $$
     \EE \zeta_m\geq C(p,d) m^{\frac{1}{a}}
    $$
    \end{enumerate} 
\end{proposition}
\begin{proof}
  We first show that $\EE \zeta_m<\infty$. If $S_{nm}=x\in B(0,m)$, then we can find a path of length less than $m$ from $x$ to a vertex $y$ with $\|y\|\geq m$, say $x_0=x,x_1,\cdots,x_k =y$. From the proof of Proposition \ref{merw01}, we see that for any $n\in \NN$,
  $$
 \PP(S_{nm+1}=x_1,S_{nm+2}=x_2\cdots,S_{nm+k}=x_k|\FF_{nm}) \mathds{1}_{\{S_{nm}=x\}}\geq c^k\mathds{1}_{\{S_{nm}=x\}} \geq c^m\mathds{1}_{\{S_{nm}=x\}}
  $$
  where $c\in (0,1)$ is a constant depending on $m$ but not on $x$. Thus,
$$\PP(\zeta_m>(n+1)m|\zeta_m>nm)\leq 1-c^m, \quad \forall n \in \NN$$ 
which implies that $\EE \zeta_m<\infty$. 

(i) We assume that $\zeta_m>4m^2$ and $S_{4m^2}=x \in B(0,m)$. Recall $\{N_n\}_{n\geq 1}$ from Lemma \ref{MartS2}. Then by Lemma \ref{MartS2} and Doob's optional sampling theorem, see e.g. Section 10.10 \cite{williams1991probability}, noting that $\{N_{ n\wedge \zeta_m}\}_{n\geq 1}$ is a martingale with bounded differences and using that for any $\ell \in [4m^2, \zeta_m-1]$,
$$
\frac{2a}{\ell}\|S_{\ell}\|^2 \geq -\frac{2 m^2}{4m^2}= -\frac{1}{2}
$$
we deduce that
\begin{equation}
  \label{Mmoptionalp}
  (m+1)^2-\|x\|^2-\frac{1}{2}(\EE \zeta_m-4m^2)\geq \EE (N_{\zeta_m}-N_{4m^2})=0
\end{equation}
and thus $\EE(\zeta_m|\FF_{4m^2})\leq 6(m+1)^2$ which yields the inequality in (1).

(ii) If $p\leq 1/(2d)$, then $a\leq 0$. Applying Doob's optional sampling theorem to $\{N_{ n\wedge \zeta_m}\}_{n\geq 1}$, we see that  
$$
\EE \zeta_m \geq  \EE \|S_{\zeta_m}\|^2 \geq m^2
$$
 
Now assume that $1/(2d)<p<p_d$, and in particular, $\{\gamma_n\}$ in Lemma \ref{MartS2} is lower bouned by 1. Recall $\{M_n\}_{n\geq 1}$ defined in Lemma \ref{MartS2}, then similarly,  $\{M_{ n\wedge \zeta_m}\}_{n\geq 1}$ is a martingale with bounded differences. By (\ref{gammandefGam}), there are positive constants $C_1(p,d)$ and $C_2(p,d)$  such that 
 \begin{equation}
  \label{c1c2gamman}
  C_1(p,d)n^{2a} \leq \gamma_n \leq C_2(p,d)n^{2a}, \quad \forall n\geq 1
 \end{equation}   
By (\ref{c1c2gamman}) and Doob's optional sampling theorem,  we can find two positive constants $C_3(p,d)$ and $C_4(p,d)$ such that for any $m\geq 1$,
  \begin{equation}
    \label{plesspdc1c2Doob}
    \EE \frac{C_3(p,d) m^2-C_4(p,d) \zeta_m}{\zeta_m^{2a}}\leq  \EE \left(\frac{m^2}{\gamma_{\zeta_m}} - \sum_{i=1}^{\zeta_m} \frac{1}{\gamma_i}\right) \leq \EE \left(\frac{\|S_{\zeta_m}\|^2}{\gamma_{\zeta_m}} - \sum_{i=1}^{\zeta_m} \frac{1}{\gamma_i}\right)=0
  \end{equation}
  By Jensen's inequality, 
  $$
  \frac{C_3(p,d) m^2-C_4(p,d)\EE \zeta_m}{(\EE\zeta_m)^{2a}} \leq \EE \frac{C_3(p,d)m^2-C_4(p,d)\zeta_m}{\zeta_m^{2a}} \leq 0, \quad \text{i.e.}\quad \EE \zeta_m\geq \frac{C_3(p,d) m^2}{C_4(p,d)}
  $$
  
(iii) and (iv): Similarly as in (\ref{plesspdc1c2Doob}), by (\ref{c1c2gamman}) and Doob's optional sampling theorem, we can find positive constants $C_5(p,d), C_6(p,d), C_7(p,d)$ and $C_8(p,d)$ such that for any $m\geq 1$,
$$
\EE \frac{C_5(p,d) m^2-C_6(p,d) \zeta_m \log \zeta_m}{\zeta_m}\leq 0, \ \text{if}\ p=p_d; \quad \EE \frac{C_7(p,d) m^2-C_8(p,d) \zeta_m^{2a}}{\zeta_m^{2a}}\leq 0, \ \text{if}\ p>p_d 
$$
It remains to apply Jensen's inequality.
\end{proof}

\subsection{Improved estimates on the rate of escape}
\label{secimprorate}

We give improved estimates on the rate of escape for MERWs on $\ZZ^d$ with $d\geq 3$ and $p\geq 1/(2d)$. We use the estimates in Proposition \ref{estexpectMERWhit} and adapt the technique used in the proof of Theorem 3.10.1 \cite{menshikov2016non}.

\begin{proof}[Proof of Proposition \ref{rateescapMERW2d}]
  The case $p=1$ is trivial. We assume that $p\in [1/(2d),1)$. Define $f(x)=\|x\|^{-1/2}$ for $x\in \ZZ^d \backslash \{0\}$. Then, for $e=e_i$, $i=1,2,\cdots,d$, 
$$
f(x+e)-f(x)=\|x\|^{-1/2}\left(\left(1+\frac{2 x \cdot e+1}{\|x\|^2}\right)^{-1/ 4}-1\right)
$$
  Using Taylor's expansion, as $\|x\|\to \infty$, we have
$$
  f(x+e)-f(x) =-\frac{1}{2}\frac{1}{\|x\|^{\frac{5}{2}}}\left(x \cdot e+\frac{1}{2}-\frac{5}{4} \frac{(x \cdot e)^2}{\|x\|^2}+O\left(\|x\|^{-1}\right)\right)
$$
Then, by Lemma \ref{computationlem} (ii), (iii), for $d\geq 3$,
$$
  \begin{aligned}
    \mathbb{E}\left[f\left(S_{n+1}\right)-f\left(S_n\right) \mid \FF_n\right] &=\frac{-1}{2\|S_n\|^{5/2}}\left(\frac{a\|S_n\|^{2}}{n}+\frac{1}{2}-\frac{5}{4}\frac{\EE([S_n \cdot \sigma_{n+1}]^2|\FF_n)}{\|S_n\|^2}+O\left(\|S_n\|^{-1}\right)\right) \\
    &\leq -\frac{1}{2\|S_n\|^{5/2}}\left(\frac{1}{2}-\frac{5}{4d}-\frac{5}{4}\sum_{i=1}^d |\frac{b_n(i)}{n}-\frac{1}{2}|+O\left(\|S_n\|^{-1}\right)\right) \\
    &\leq -\frac{1}{2\|S_n\|^{5/2}}\left(\frac{1}{20}-\frac{5}{4}\sum_{i=1}^d |\frac{b_n(i)}{n}-\frac{1}{2}|\right)
  \end{aligned}
$$
if $S_n \notin B(0,r)$ for some large $r$. Recall $\zeta_m$ defined in (\ref{exittimedef}). For $m_2>m_1 \in \NN\backslash \{0\} $, let $\lambda_{m_2, m_1}:=\inf \left\{n \geq \tau_{m_2}: \|S_n\| \leq m_1\right\}$ and for $k\geq 1$, let
$$
T_k:=\inf\{j \geq k:  \ \sum_{i=1}^d|\frac{b_j(i)}{j}-\frac{1}{2}|\geq \frac{1}{j^{1/4}}\}
$$ 
We assume that $m_1>r$, then by possibly choosing a larger $r$, we see that
$$
\{f(S_{(\tau_{m_2}+n) \wedge T_{\tau_{m_2}}\wedge \lambda_{m_2, m_1}})\}_{n\in \NN}
$$
is a non-negative supermartingale. Then, similarly as in the proof of (\ref{inequtkcapTk}), by the optional stopping theorem for non-negative supermartingales, for any $m_1>r$, we have
\begin{equation}
  \label{lammim2bd}
  \begin{aligned}
   \frac{1}{\sqrt{m_2}}&\geq f(S_{\tau_{m_2}}) \geq \EE (f(S_{\lambda_{m_2, m_1}})\mathds{1}_{\{\lambda_{m_2, m_1}<\infty\} \cap \{T_{\tau_{m_2}}=\infty\}}|\FF_{\tau_{m_2}}) \\
   &\geq  \frac{1}{\sqrt{m_1}}\mathbb{P}\left(\lambda_{m_2, m_1}<\infty, T_{\tau_{m_2}}=\infty \mid \mathcal{F}_{\tau_{m_2}}\right) 
\end{aligned}
\end{equation}
For any $x\in \RR_+$, define
$$
\eta_x:=\sup \left\{n \geq 0: \|S_n\| \leq x\right\}
$$
Note that $\eta_x$ is not a stopping time. Observe that by (\ref{lammim2bd})
\begin{equation}
  \label{3101etaT}
  \begin{aligned}
  \mathbb{P}\left(\eta_{m_1}>n, T_{\tau_{m_2}}=\infty\right) &\leq \mathbb{P}(\eta_{m_1}>n, \tau_{m_2}\leq n, T_{\tau_{m_2}}=\infty)+\mathbb{P}(\tau_{m_2}> n, T_{\tau_{m_2}}=\infty) \\
  &\leq  \mathbb{P}(\lambda_{m_2, m_1}<\infty, T_{\tau_{m_2}}=\infty)+\mathbb{P}(\tau_{m_2}>n)\\ &\leq \sqrt{\frac{m_1}{m_2}}+\frac{6(1+m_2)^2}{n}
\end{aligned}
\end{equation}
where we used Proposition \ref{estexpectMERWhit} and Markov's inequality in the last inequality. Fix $\varepsilon>0$ and set
$$
m_1(k)=2^k,\quad m_2(k)=2^k(\log 2^k)^{2+\varepsilon}, \quad  n(k)=4^k(\log 2^k)^{5+4 \varepsilon}
$$
Then, for all $k$ sufficiently large, by (\ref{3101etaT})
$$
\mathbb{P}(\eta_{2^k}>n(k),T_{\tau_{m_2(k)}}=\infty) \leq \frac{1}{(k \log 2)^{1+\frac{1}{2} \varepsilon}} +\frac{12}{(k \log 2)^{1+\varepsilon}}
$$
So $\sum_{k \in \NN} \mathbb{P}\left[\eta_{2^k}>n(k),T_{\tau_{m_2(k)}}=\infty\right]<\infty$. The Borel-Cantelli lemma shows that, a.s., $\{\eta_{2^k}>n(k),T_{\tau_{m_2(k)}}=\infty\}$, $k=1,2,3\cdots $, occur finitely often. By Lemma \ref{ratecon1overd}, a.s. $\{T_{\tau_{m_2(k)}}<\infty\}$, $k=1,2,3\cdots $, occur finitely often. Therefore, $\eta_{2^k} \leq n\left(k\right)$ for all but finitely many $k \in \NN_{+}$. Thus, almost surely, for all $x \in \RR_+$ sufficiently large, 
$$
\eta_x \leq \eta_{2^{\lfloor \log_2 x \rfloor+1}} \leq n\left( \lfloor \log_2 x \rfloor+1 \right) \leq 4x^2 (\log (2x))^{5+4 \varepsilon} \leq x^2(\log x)^{5+5\varepsilon}
$$
where the floor function $\lfloor \cdot \rfloor: \RR \to \ZZ$ is defined by $\lfloor y\rfloor=\max \{m \in \mathbb{Z} \mid m \leq y\}$. Since $\|S_n\| \rightarrow \infty$ by Theorem \ref{MERWtran12d}, we deduce that, a.s., for all but finitely many $n$, 
\begin{equation}
  \label{nSnlogSn5}
  n \leq \eta_{\|S_n\|} \leq \|S_n\|^2\left(\log \|S_n\|\right)^{5+5\varepsilon}
\end{equation}
We may choose $\varepsilon< 1/5$. If $\|S_n\| \leq n^{\frac{1}{2}}(\log n)^{-3}$, then 
$$
\|S_n\|^2\left(\log \|S_n\|\right)^{5+5\varepsilon} \leq \frac{n (\log n)^{5+5\varepsilon} }{(\log n)^{6}} 
$$
which can occur for only finitely many $n$ in view of (\ref{nSnlogSn5}).
\end{proof}

\subsection{The Berry-Esseen type bounds for the ERW}
\label{rateconclteswp0}

Let $(S_n)_{n\geq 0}$ be an elephant random walk with parameter $p\in [0,1]$. For $n\geq 1$, we let 
\begin{equation}
  \label{mnhatandef}
  \bar{S}_n:=a_n S_n, \quad \text{where}\ a_n:= \prod_{k=1}^{n-1}\frac{k}{k+2 p-1} =\frac{\Gamma(n) \Gamma(2 p)}{\Gamma(n+2 p-1)}
\end{equation}
with the convention that $a_1=1$ (and thus $\bar{S}_1=1$). Then by (\ref{Esigman1}), $(\bar{S}_n)_{n\geq 1}$ is a martingale, as was pointed out by Bercu in \cite{bercu2017martingale}. Note that $a_n \sim \Gamma(2p)n^{1-2p}$. To prove Theorem \ref{cltp}, we need the following auxiliary lemma.
\begin{lemma}
  \label{estimatean}
  There exists a positive constant $C(p)$ such that for any $n\geq 1$
  \begin{equation}
    \label{andiffn12p}
    |a_{n+1}-\Gamma(2p)n^{1-2p}| \leq \frac{C(p)}{n^{2p}}
  \end{equation}
  Moreover, for $p<3/4$, resp. $p=3/4$, there exists a positive constant $C_1(p)$, resp. a positive constant $C$ such that
  \begin{equation}
    \label{nan2bd}
    |\frac{na_n^2}{3-4p} - \sum_{k=1}^na_k^2| \leq C_1(p)n^{\max(2-4p,0)}, \quad \text{resp.}\quad |na_n^2\log n - \sum_{k=1}^na_k^2| \leq C
  \end{equation}
\end{lemma}
\begin{remark}
  If $p=1/2$, $a_n \equiv \Gamma(2p)=1$.
\end{remark}
\begin{proof}
  By Stirling’s asymptotic series, see e.g. Section VII \cite{spiegel2018mathematical}, for $x>0$,
\begin{equation}
  \label{Stirlingasy}
   \Gamma(x+1)=\sqrt{2 \pi x} x^x e^{-x}(1+O(\frac{1}{x}))
\end{equation}
By (\ref{Stirlingasy}) and Taylor expansion, for any $n\geq 1$,
$$
\begin{aligned}
  &\quad n^{2p-1}|a_{n+1}-\Gamma(2p)n^{1-2p}| =\Gamma(2p) \left|\frac{ \sqrt{n}n^{n+2p-1} e^{-n}(1+O(\frac{1}{n})) }{ \sqrt{n+2p-1}(n+2p-1)^{n+2p-1} e^{-n-2p+1}(1+O(\frac{1}{n})) }-1\right| \\
  &=\Gamma(2p)\left|\sqrt{\frac{1}{1+\frac{2p-1}{n}}}\exp(n(1+\frac{2p-1}{n})\log (\frac{1}{1+\frac{2p-1}{n}})+2p-1) -1 +O(\frac{1}{n})\right| \leq \frac{C(p)}{n}
\end{aligned}
$$
for some positive constant $C(p)$, which completes the proof of (\ref{andiffn12p}). Thus, for any $p\in [0,1]$, there exists a positive constant $C_2(p)$ such that for any $n\geq 1$,
\begin{equation}
  \label{an2bd}
  |a_{n+1}^2-\Gamma(2p)^2n^{2-4p}|\leq C_2(p)n^{1-4p}
\end{equation}
from which one can easily deduce (\ref{nan2bd}).
\end{proof}

\begin{proof}[Proof of Theorem \ref{cltp}]
 Let $\bar{S}$ be the martingale defined in (\ref{mnhatandef}). For any $p\in [0,3/4]$, by Theorem 3 (or Remark 2) in \cite{MR4646949} (note that one can apply the same argument there to prove the case $p=0$) and the Lipschitz property of $\Phi$, we can find a positive constant $C_3(p)$ such that for any $n\geq 1$, 
 \begin{equation}
  \label{barSnphibd}
     \sup_{t \subset \mathbb{R}}\left|\mathbb{P}\left(\frac{ \bar{S}_{n}}{\sqrt{ \sum_{k=1}^na_k^2}} \leq t\right)-\Phi(t)\right| \leq C_3(p) \left( \frac{1}{\sqrt{n}} + \frac{1}{\sqrt{\sum_{k=1}^na_k^2}}\right) 
 \end{equation}

 \textbf{Case 1:} $p<3/4$. By Lemma \ref{estimatean}, there exists a positive constant $C_4(p)$ such that the right-hand side of (\ref{barSnphibd}) is bounded by $C_4(p)n^{-\min(1,3-4p)/2}$ for all $n\geq 1$. For any $t\in \RR$, since $\bar{S}_n=a_nS_n$, by (\ref{barSnphibd}),
 \begin{equation}
  \label{34pSnphibd}
    \begin{aligned}
   |\PP(\frac{\sqrt{3-4p}S_n}{\sqrt{n}} \leq t)-\Phi(t)| &= |\PP(\frac{\bar{S}_n}{\sqrt{ \sum_{k=1}^na_k^2}} \leq \frac{\sqrt{\frac{n}{3-4p}}a_nt}{\sqrt{ \sum_{k=1}^na_k^2}})-\Phi(t)| \\
   &\leq \frac{C_4(p)}{n^{\min(1,3-4p)/2}} + |\Phi(\frac{\sqrt{\frac{n}{3-4p}}a_nt}{\sqrt{ \sum_{k=1}^na_k^2}})- \Phi(t)|
 \end{aligned}
 \end{equation}
 Now we provide an upper bound of the last term on the right-hand side. By (\ref{nan2bd}),
there exists a constant $C_5(p)$ such that for any $n\geq 1$,
 \begin{equation}
  \label{anMn1bd}
     |\frac{\sqrt{\frac{n}{3-4p}}a_n}{\sqrt{ \sum_{k=1}^na_k^2}}- 1|=|\frac{\frac{na_n^2}{3-4p}-\sum_{k=1}^na_k^2}{\sqrt{ \sum_{k=1}^na_k^2}(\sqrt{\frac{n}{3-4p}}a_n+\sqrt{ \sum_{k=1}^na_k^2})}| \leq \frac{C_5(p)}{n^{\min(1,3-4p)}}
 \end{equation}
 and thus, if $|t|\leq n^{\min(1,3-4p)/2}$,
 $$
 |\frac{\sqrt{\frac{n}{3-4p}}a_nt}{\sqrt{ \sum_{k=1}^na_k^2}}- t|\leq \frac{C_5(p)|t|}{n^{\min(1,3-4p)}} \leq \frac{C_5(p)}{n^{\min(1,3-4p)/2} }
 $$
 The Lipschitz property of $\Phi$ then implies that the last term in (\ref{34pSnphibd}) is upper bounded by $C_5(p)n^{-\min(1,3-4p)/2}$. If $|t|> n^{\min(1,3-4p)/2}$, the last term in (\ref{34pSnphibd}) is still $O(n^{-\min(1,3-4p)/2})$ uniformly in $t$ by the properties of $\Phi$ (i.e. $\Phi(-x) \to 0$ and $1-\Phi(x) \to 0$ exponentially fast as $x \to \infty$), which completes the proof of Case 1.

 \textbf{Case 2}: $p=3/4$. By Lemma \ref{estimatean} and (\ref{barSnphibd}), similarly as in (\ref{34pSnphibd}), for some constant $C_1$, we have for all $n>1$,
 \begin{equation}
   \label{Snphi34}
     |\PP(\frac{S_n}{\sqrt{n\log n}} \leq t)-\Phi(t)| \leq \frac{C_1}{\sqrt{\log n}} + |\Phi(\frac{\sqrt{n\log n}a_nt}{\sqrt{\sum_{k=1}^na_k^2 }})- \Phi(t)|
 \end{equation}
 Now by (\ref{nan2bd}), as in (\ref{anMn1bd}), one can show that there exists a constant $C_2$ such that for any $n> 1$ and $t\in \RR$,
 \begin{equation}
   \label{sqrtantMnt}
   |\frac{\sqrt{n\log n}a_nt}{\sqrt{\sum_{k=1}^na_k^2 }}- t|\leq \frac{C_3|t|}{\log n}
 \end{equation}
 The rest of the proof follows the same lines as that of Case 1: We consider the two cases $|t|\geq \sqrt{\log n}$ and $|t|\leq \sqrt{\log n}$, instead. For the former case, we use the exponential decay of $\Phi$. For the latter one, we use (\ref{sqrtantMnt}) and the Lipschitz property of $\Phi$. In either case, one can show that the last term in (\ref{Snphi34}) is $O((\log n)^{-1/2})$ uniformly in $t$. 
\end{proof}

\section{Acknowledgement}

I am very grateful to Professor Tarrès, my Ph.D. advisor, for inspiring the choice of this subject and for helpful discussions and advice. Thanks also to Zheng Fang for careful reading.

I would also like to thank LPSM at Sorbonne Université for hosting me during Spring 2023 when work on this paper was undertaken. The visit to Paris was supported by the NYU-ECNU Institute funding.

\bibliographystyle{plain}
\bibliography{math_ref}

\end{document}